\documentclass[times,sort&compress,3p]{elsarticle}
\journal{Journal of Multivariate Analysis}

\usepackage[utf8]{inputenc}

\usepackage{amsmath}
\usepackage{amsfonts}
\usepackage{amssymb}
\usepackage{amsthm}
\usepackage{paralist}
\usepackage{booktabs}

\usepackage{natbib}
\usepackage{color}
\usepackage[colorlinks=true,linkcolor=blue,citecolor=blue,pdfborder={0 0 0}]{hyperref}

\usepackage{dsfont}
\usepackage{bm}
\usepackage{enumitem}
\usepackage{graphicx}
\usepackage{enumerate}
\usepackage{epstopdf}

\newtheorem{theorem}{Theorem}

\newtheorem{lemma}[theorem]{Lemma}
\newtheorem{corollary}[theorem]{Corollary}

\theoremstyle{definition}

\newtheorem{condition}[theorem]{Condition}

\theoremstyle{remark}
\newtheorem{remark}[theorem]{Remark}

\newcommand{\dto}{\rightsquigarrow}

\newcommand{\R}{\mathbb{R}}

\newcommand{\Z}{\mathbb{Z}}

\newcommand\Prob{\mathbb{P}}    
\newcommand{\Cb}{\mathbb{C}}
\newcommand{\Bin}{\operatorname{Bin}}
\newcommand{\oh}{\mathrm{o}}
\newcommand{\Oh}{\mathrm{O}}
\newcommand{\point}{\,\cdot\,}
\renewcommand{\log}{\ln} 

\newcommand{\vect}{\bm }  

\newcommand{\eps}{\varepsilon}
\newcommand{\ind}{\operatorname{\mathds{1}}}
\newcommand{\diff}{{\,\mathrm{d}}}

\renewcommand{\leq}{\leqslant}
\renewcommand{\le}{\leqslant}
\renewcommand{\geq}{\geqslant}
\renewcommand{\ge}{\geqslant}

\newcommand{\Exp}{\operatorname{E}}
\newcommand{\Var}{\operatorname{Var}}
\newcommand{\Cov}{\operatorname{Cov}}

\newcommand{\Acfg}{\hat{A}_n^{\mathrm{CFG}}} 
\newcommand{\Acfgc}{\hat{A}_{n,\mathrm{c}}^{\mathrm{CFG}}} 
\newcommand{\Acfgb}{\hat{A}_{n,\beta}^{\mathrm{CFG}}} 

\newcommand{\Cn}{\hat{C}_n} 
\newcommand{\Cnb}{C_n^{\beta}}

\definecolor{auburn}{rgb}{0.43, 0.21, 0.1}
\definecolor{britishracinggreen}{rgb}{0.0, 0.26, 0.15}
\definecolor{burntumber}{rgb}{0.54, 0.2, 0.14}
\definecolor{carmine}{rgb}{0.59, 0.0, 0.09}

\newcommand{\johan}[1]{ \textcolor{red}{\sffamily\small [J: {#1}]}}

\begin{document}

\begin{frontmatter}
\title{Weak convergence of the weighted empirical beta copula process}

\author[rub]{Betina Berghaus}
\ead{betina.berghaus@rub.de}
\address[rub]{Ruhr-Universit\"at Bochum, Fakult\"at f\"ur Mathematik, Universit\"atsstr.\ 150, 44780 Bochum, Germany}

\author[ucl]{Johan Segers}
\ead{johan.segers@uclouvain.be}
\address[ucl]{Universit\'e catholique de Louvain, ISBA, Voie du Roman Pays 20, 1348 Louvain-la-Neuve, Belgium}


\begin{abstract}
The empirical copula has proved to be useful in the construction and understanding of many statistical procedures related to dependence within random vectors. The empirical beta copula is a smoothed version of the empirical copula that enjoys better finite-sample properties. At the core lie fundamental results on the weak convergence of the empirical copula and empirical beta copula processes. Their scope of application can be increased by considering weighted versions of these processes. In this paper we show weak convergence for the weighted empirical beta copula process.  The weak convergence result for the weighted empirical beta copula process is stronger than the one for the empirical copula and its use is more straightforward. The simplicity of its application is illustrated for weighted Cram\'er--von Mises tests for independence and for the estimation of the Pickands dependence function of an extreme-value copula. 
\end{abstract}

\end{frontmatter}

\section{Introduction}

In many statistical questions related to multivariate dependence, a crucial role is played by the copula function. 
A basic nonparametric copula estimator is the empirical copula, 
dating back to \cite{Rus76, Deh79} and defined as the empirical distribution function of the vectors of component-wise ranks.
The asymptotic behavior of the empirical copula has been established under various assumptions on the true copula and the serial dependence of the observed random vectors \citep[see, e.g.,][]{GanStu87,FerRadWeg04,Seg12,BucVol13}. The upshot is that the empirical copula process converges weakly to a centered Gaussian field with covariance function depending on the true copula and the serial dependence of the observations.

Recently, \citet{BerBucVol17} investigated the weak convergence of the weighted empirical copula process. They showed that the empirical copula process divided by a weight function, that can be zero on parts of the boundary of the unit cube, still converges weakly to a Gaussian field.  As illustrated in the latter reference, this stronger result allows for additional applications of the continuous mapping theorem or the functional delta method. However, this result is only valid for a clipped version  of the process. Since the empirical copula itself is not a copula, weak convergence fails on the upper boundaries of the unit cube \citep[Remark~2.3]{BerBucVol17}. 

The empirical beta copula \citep{SegSibTsu17} arises as a particular case of the empirical Bernstein copula \citep[see, e.g.,][]{SanSat04,JanSwaVer12} if the degrees of the Bernstein polynomials are set to the sample size. In the numerical experiments in \citep{SegSibTsu17}, the empirical beta copula exhibited a better performance than the empirical copula, both in terms of bias and variance. 

In contrast to the empirical copula, the empirical beta copula is a genuine copula, a property that it shares with the checkerboard copula, whose limit is derived in \cite{GenNesRem2017} and which is very close to the empirical copula if the margins are continuous. Since the empirical beta copula is itself a copula, it is possible to prove weighted weak convergence for the empirical beta copula process on the whole unit cube. This is the main result of the paper. Weak convergence on the whole unit cube rather than on a subset thereof is quite handy since it allows for a direct application of, e.g., the continuous mapping theorem. In particular, there is no longer any need to treat the boundary regions separately. 

We consider two applications. First, we modify the Cram\'er--von Mises test statistic for independence in \cite{genest+r:2004} by using the empirical beta copula and, more importantly, adding a weight function in the integral, emphasizing the tails. The asymptotic distribution of the statistic under the null hypothesis is an easy corollary of our main result. More interestingly, the inclusion of a weight function leads to a markedly better power against difficult alternatives such as the \emph{t} copula with zero correlation parameter, with favorable comparisons even to the novel statistics introduced recently by \citet{belalia+b+l+t:2017}. As a second application we consider the Cap\'er\`aa--Foug\`eres--Genest estimator \citep{CapFouGen97} of the Pickands dependence function of a multivariate extreme-value copula. Under weak dependence, replacing the empirical copula by the empirical beta copula yields a more accurate estimator. Its asymptotic distribution is again an immediate consequence of our main result.

The paper is organized as follows. In Section~\ref{sec:main} we introduce the various empirical copula processes and we state the main result of the paper, the weighted convergence of the empirical beta copula process on the whole unit cube. We illustrate the ease of application of the main result to the analysis of weighted Cram\'er--von Mises tests of independence (Section~\ref{sec:indep}) and nonparametric estimation of multivariate extreme-value copulas (Section~\ref{sec:pick}). The proofs are deferred to Section~\ref{sec:proof}, whereas a number of technical arguments are worked out in detail in Section~\ref{sec:aux}.

\section{Notation and main result}
\label{sec:main}

Let $(\vect X_n)_n$ be a strictly stationary time series whose $d$-variate stationary distribution function $F$ has continuous marginal distribution functions $F_1,\dots, F_d$ and copula $C$. Writing $\vect X_i = (X_{i,1}, \ldots, X_{i,d})$, we have, for $\vect x \in \R^d$,
\begin{align*}
  \Prob( X_{i,j} \le x_j ) &= F_j(x_j), &
  \Prob( \vect X_i \le \vect x ) &= F(\vect x) = C\{ F_1(x_1), \ldots, F_d(x_d) \}.
\end{align*}
For vectors $\vect x, \vect y \in \R^d$, the inequality $ \vect x \leq \vect y$ means that $x_j \leq y_j$ for $j = 1, \dots, d$. Similar conventions apply for other inequalities and for minima and maxima, denoted by the operators $\wedge$ and $\vee$, respectively.
 Given the sample $\vect X_1, \ldots, \vect X_n$, the aim is to estimate $C$ and functionals thereof.

Although the copula $C$ captures the instantaneous (cross-sectional) dependence, the setting is still general enough to include questions about serial dependence. For instance, if $(Y_n)_{n}$ is a univariate, strictly stationary time series, then the $d$-variate time series of lagged values $\vect X_n = (Y_n, Y_{n-1}, \ldots, Y_{n-d+1})$ is strictly stationary too and the instantaneous dependence within the series $(\vect X_n)_n$ corresponds to serial dependence within the original series $(Y_n)_n$ up to lag $d-1$.

For $i = 1, \ldots, n$ and $j = 1, \ldots, d$, let $R_{i,j}$ denote the rank of $X_{i,j}$ among $X_{1,j}, \ldots, X_{n,j}$. For convenience, we omit the sample size $n$ in the notation for ranks. The random vectors $\hat{\vect U}_i = (\hat{U}_{i,1}, \ldots, \hat{U}_{i,d})$, with $\hat{U}_{i,j} = n^{-1} R_{i,j}$ and $i = 1, \ldots, n$, are called pseudo-observations from $C$. Letting $\ind_A$ denote the indicator of the event $A$, the empirical copula is
\[
  \hat C_n(\vect u) = \frac{1}{n} \sum_{i=1}^n \ind_{ \{ \hat{\vect U}_i \leq \vect u\} }, \qquad \vect u \in [0,1]^d.
\]

Under mixing conditions on the sequence $( \vect X_n )_n$ and smoothness conditions on $C$, \citet{BucVol13} showed that
\begin{equation}
\label{eq:hCn:weak}
  \hat{\Cb}_n = \sqrt n(\hat C_n - C) \dto \Cb_C, \qquad n \to \infty
\end{equation}
in the metric space $\ell^\infty([0,1]^d) = \{ f : [0, 1]^d \to \R \mid \sup_{\vect u \in [0, 1]^d} \lvert f(\vect u) \rvert < \infty \}$ equipped with the supremum distance. The arrow $\dto$ in \eqref{eq:hCn:weak} denotes weak convergence in metric spaces as exposed in \cite{VanWel96}. The limit process in \eqref{eq:hCn:weak} is
\[
  \Cb_C(\vect u) 
  = \alpha_C(\vect u) - \sum_{j=1}^d \dot C_j(\vect u) \, \alpha_C(1, \dots, 1, u_j, 1, \dots, 1),
  \qquad \vect u \in [0, 1]^d,
\]
where $\dot{C}_j(\vect u) = \partial C(\vect u) / \partial u_j$ and where $\alpha_C$ is a tight, centered Gaussian process on $[0, 1]^d$ with covariance function
\begin{equation}
\label{eq:cov}
  \Cov\bigl( \alpha_C(\vect u), \alpha_C(\vect v) \bigr) 
  = \sum_{i \in \Z} \Cov\bigl( \ind_{ \{ \vect U_0 \leq \vect u \} }, \ind_{ \{ \vect U_i \leq \vect v \} } \bigr),
  \qquad \vect u, \vect v \in [0, 1]^d,
\end{equation}
where $\bm{U}_i = (U_{i,1}, \ldots, U_{i,d})$ and $U_{i,j} = F_j(X_{i,j})$. Since $F_j$ is continuous, the random variables $U_{i,j}$ are uniformly distributed on $[0, 1]$. The joint distribution function of $\bm{U}_i$ is $C$. The margins $F_1, \ldots, F_d$ being unknown, we cannot observe the $\vect U_i$, and this is why we use the $\hat{\vect U}_i$ instead. In the case of serial independence, weak convergence of $\hat{\Cb}_n$ has been investigated by many authors, see the survey by \citet{BucVol13}; the series in \eqref{eq:cov} simplifies to $\Cov( \ind_{ \{\vect U_0 \le \vect u\} }, \ind_{ \{ \vect U_0 \le \vect v \} } ) = C( \vect u \wedge \vect v ) - C( \vect u) C( \vect v)$ so that $\alpha_C$ is a $C$-Brownian bridge. In the stationary case, convergence of the series in \eqref{eq:cov} is a consequence of the mixing conditions imposed on $(\vect X_n)_n$. 

Weak convergence in \eqref{eq:hCn:weak} is helpful for deriving asymptotic properties of estimators and test statistics based upon the empirical copula, such as estimators of Kendall's tau or Spearman's rho or such as Kolmogorov--Smirnov and Cram\'er--von Mises statistics for testing independence. However, as argued by \citet{BerBucVol17}, sometimes weak convergence with respect to a stronger metric is required, i.e., a weighted supremum norm. Examples mentioned in the cited article include nonparametric estimators of the Pickands dependence function of an extreme-value copula and bivariate rank statistics with unbounded score functions such as the van der Waerden rank (auto-)correlation. This motivates the study of the weighted empirical copula process $\hat {\Cb}_n/g^\omega$, with $\omega \in(0,1/2)$ and a suitable weight function $g$ on $[0, 1]^d$. The limit of the empirical copula process is zero almost surely as soon as one of its arguments is zero or if all arguments but at most one are equal to one. We can thus hope to obtain weak convergence with respect to a weight function that vanishes at such points. A possible function with this property is
\begin{equation}
\label{eq:g}
  g(\vect u ) =
  \bigwedge_{j=1}^d \biggl\{ u_j \wedge \bigvee_{k \ne j} (1 - u_k) \biggr\},
  \qquad \vect u \in [0, 1]^d.
\end{equation}
Note that $g(\vect u)$ is small as soon as there exists $j$ such that either $u_j$ is small or else all other $u_k$ are close to $1$.
The trajectories of the processes $\hat{\Cb}_n/g^\omega$ are not bounded on the unit cube, hence the processes cannot converge weakly in $\ell^\infty([0,1]^d)$. A solution is to restrict the domain from $[0, 1]^d$ to sets of the form $[c/n, 1 - c/n]^d$ for $c \in (0, 1)$, or, more generally, to $\{ \vect v \in [0, 1]^d : g(\vect v) \ge c/n \}$. Relying on such a workaround, Theorem~2.2 in \cite{BerBucVol17} states weak convergence of the weighted empirical copula process $\hat{\Cb}_n/g^\omega$ to $\Cb_C/g^\omega$.
Note that $g(\vect v) = 0$ if and only if $v_j = 0$ for some $j$ or if there exists $j$ such that $v_k = 1$ for all $k \ne j$, and that $\Cb_C(\vect v) = 0$ almost surely for such $\vect v$ too.

The empirical copula is a piecewise constant function whereas the estimation target is continuous. It is natural to consider smoothed versions of the empirical copula. \citet{SegSibTsu17} defined the empirical beta copula as
\begin{equation}
\label{eq:empBetaCop}
  \Cnb (\vect u )= \frac{1}{n} \sum_{i=1}^n\prod_{j=1}^d F_{n,R_{i,j}}(u_j), 
  \qquad \vect u =(u_1,\dots , u_d) \in [0,1]^d,
\end{equation}
where $F_{n,r}$ is the distribution function of the beta distribution $\mathcal{B}(r,n+1-r)$, i.e., $F_{n,r}(u) = \sum_{s=r}^n \binom{n}{s} u^s (1-u)^{n-s}$, for $u \in [0, 1]$ and $r \in \{1, \ldots, n\}$. Note that 
\begin{equation}
\label{eq:empCop2empBetaCop}
  \Cnb (\vect u ) = \int_{[0,1]^d} \hat C_n (\vect w) \diff \mu_{n, \vect u}(\vect w),
\end{equation}
where $\mu_{n,\vect u}$ is the law of the random vector $(S_1/n, \dots , S_d/n)$, with $S_1, \dots, S_d$ being independent binomial random variables, $S_j \sim \Bin(n, u_j)$. In the absence of ties, the rank vector $(R_{1,j}, \ldots, R_{n,j})$ of the $j$-th coordinate sample is a permutation of $(1, \ldots, n)$. As a consequence, the empirical beta copula can be shown to be a genuine copula, unlike the empirical copula.

Under a smoothness condition on $C$, it follows from Theorem~3.6(ii) in \cite{SegSibTsu17} that weak convergence in $\ell^\infty([0, 1]^d)$ of the empirical copula process $\hat{\Cb}_n$ in \eqref{eq:hCn:weak} to a limit process $\Cb$ with continuous trajectories is sufficient to conclude the weak convergence of the empirical beta copula process: in the space $\ell^\infty([0, 1]^d)$, we have
\begin{equation}
\label{eq:betacop:weak}
  \Cb_n^{\beta} = \sqrt n (\Cnb - C) = \hat{\Cb}_n + \oh_{\Prob}(1) \dto \Cb, \qquad n \to \infty.
\end{equation}
The asymptotic distribution of the empirical beta copula is thus the same as the one of the empirical copula. Still, for finite samples, numerical experiments in \cite{SegSibTsu17} revealed the empirical beta copula to be more accurate.

Our aim is to extend the convergence statement in \eqref{eq:betacop:weak} for weighted versions $\Cb_n^{\beta} / g^\omega$, with $g$ as in \eqref{eq:g} and for suitable exponents $\omega > 0$. As the empirical beta copula is a genuine copula, the zero-set of $\Cb_n^{\beta}$ includes the zero-set of $g$, and on this set we implicitly define $\Cb_n^{\beta} / g^\omega$ to be zero. With this convention, the sample paths of $\Cb_n^{\beta} / g^\omega$ are bounded on $[0, 1]^d$; see Lemma~\ref{lem:boundary} below. We can therefore hope to prove weak convergence of $\Cb_n^{\beta} / g^\omega \dto \Cb_C/g^\omega$ in $\ell^\infty([0, 1]^d)$ without having to exclude those border regions of $[0, 1]^d$ where $g$ is small, as was necessary in \cite{BerBucVol17}. 

The analysis of $\Cb_n^{\beta}/g^\omega$ will be based on the one of $\hat{\Cb}_n/g^\omega$ via \eqref{eq:empCop2empBetaCop}. We will therefore need the same smoothness condition on $C$ as imposed in \citet[Condition~2.1]{BerBucVol17}, combining Conditions~2.1 and~4.1 in~\cite{Seg12}. Condition~\ref{cond:second} below is satisfied by many copula families: in \citep[Section~5]{Seg12}, part~(i) of the condition is verified for strict Archimedean copulas with continuously differentiable generators, whereas both parts of the condition are verified for the non-singular bivariate Gaussian copula and for bivariate extreme-value copulas with twice continuously differentiable Pickands dependence function and a growth condition on the latter's second derivative near the boundary points of its domain.


\begin{condition}
\label{cond:second}
(i) For every $j \in \{ 1, \dots, d \}$, the first-order partial derivative $\dot C_j(\vect u) := \partial C(\vect u)/\partial u_j$ exists and is continuous on $V_j=\{ \vect u \in [0,1]^d: u_j \in (0,1) \}$.

(ii) For every $j_2, j_2 \in \{1, \dots, d\}$, the second-order partial derivative $\ddot C_{j_1 j_2}(\vect u) := \partial^2 C(\vect u)/\partial u_{j_1}\partial u_{j_2}$ exists and is continuous on $V_{j_1} \cap V_{j_2}$.  Moreover, there exists a constant $K>0$ such that, for all $j_1, j_2 \in \{1, \ldots, d\}$, we have
\begin{equation}
\label{eq:second}
  \bigl\lvert \ddot C_{j_1j_2}(\vect u) \bigr\rvert 
  \le  K \min \left\{ \frac{1}{u_{j_1}(1-u_{j_1})}, \frac{1}{u_{j_2}(1-u_{j_2})} \right\}, \qquad \forall\, \vect u \in V_{j_1} \cap V_{j_2}.
\end{equation}
\end{condition}

The alpha-mixing coefficients of the sequence $(\vect X_n)_n$ are defined as
\[
  \alpha(k) = 
  \sup \left\{ 
    \lvert \Prob(A \cap B) - \Prob(A) \, \Prob(B) \rvert : 
    A \in \sigma(\vect X_j, j \le i), B \in \sigma(\vect X_{j+k}, j \ge i), i \in \Z 
  \right\},
\]
for $k = 1, 2, \ldots$. The sequence $(\vect X_n)_n$ is said to be strongly mixing or alpha-mixing if $\alpha(k) \to 0$ as $k \to \infty$. Now we can state the main result. 

\begin{theorem}
\label{thm:main}
Suppose that $\vect X_1, \vect X_2, \dots$ is a strictly stationary, alpha-mixing sequence with $\alpha(k) = \Oh(a^k)$, as $k \to \infty$, for some $a \in (0,1)$. Assume that within each variable, ties do not occur with probability one. If the copula $C$ satisfies Condition~\ref{cond:second}, then, for any $\omega\in[0,1/2)$, we have, in $\ell^\infty([0, 1]^d)$,
\[
  \Cb_n^\beta/g^\omega \dto \Cb_C/g^\omega, \qquad n \to \infty.
\]
\end{theorem}

\begin{remark}
The tie-excluding assumption is needed to ensure that the empirical beta copula is a genuine copula almost surely. The assumption implies that the $d$ stationary marginal distributions are continuous. For iid sequences, continuity of the margins is also sufficient. In the strictly stationary case, ties may occur with positive probability even if the margins are continuous; for instance, take a Markov chain where the current state is repeated with positive probability.
\end{remark}

\begin{remark}
The result also holds under weaker assumptions on the serial dependence. In \cite{BerBucVol17} it is shown that weak convergence of the weighted empirical copula process is still valid under more general assumptions on the marginal empirical processes and quantile processes and an assumption on the multivariate empirical process. In this case, however, the range of $\omega$ is smaller \citep[Theorem~4.5]{BerBucVol17}.
 \end{remark}

\section{Application: weighted Cram\'er--von Mises tests for independence}
\label{sec:indep}

Testing for independence is a classical subject which still attracts interest today. One approach consists of comparing the multivariate empirical cumulative distribution function to the product of empirical cumulative distribution functions. Integrating out the difference with respect to the sample distribution yields a Cram\'er--von Mises style test statistic going back to \citet{hoeffding:1948} and \citet{blum+k+r:1961}. To achieve better power, one may, in the spirit of the Anderson--Darling goodness-of-fit test statistic, introduce a weight function in the integral that tends to infinity near (parts of) the boundary of the domain; see \citet{dewet:1980}.

\citet{deheuvels:1981, deheuvels:1981:jmva} was perhaps the first to reformulate the question in the copula framework: for continuous variables, the problem consists in testing whether the true copula, $C$, is equal to the independence copula, $\Pi(\vect u) = \prod_{j=1}^d u_j$. The empirical copula process $\sqrt{n} (\Cn - \Pi)$, for which he proposed an ingenious combinatorial transformation, can thus be taken as a basis for the construction of test statistics. \citet{genest+r:2004} relied on his ideas to test the white noise hypothesis and considered Cram\'er--von Mises statistics based on the empirical copula process. \citet{genest+q+r:2006} studied the power of such statistics against local alternatives, while \citet{kojadinovic+h:2009} developed an extension to the case of testing for independence between random vectors. For the latter problem, \citet{fan:2017} proposed an alternative approach based on empirical characteristic functions.

Recently, \citet{belalia+b+l+t:2017} proposed to use the Bernstein empirical copula \cite{SanSat04,JanSwaVer12} rather than the empirical copula in the Cram\'er--von Mises test statistic. Moreover, they constructed new test statistics based on the Bernstein copula density estimator by \citet{bouezmarni+r+t:2010}. Recall that the empirical beta copula arises from the Bernstein empirical copula by a specific choice of the degree of the Bernstein polynomials.

A situation of particular interest is when the true copula differs from the independence copula mainly in the tails. For instance, the bivariate \emph{t} copula with zero correlation parameter has both Spearman's rho and Kendall's tau equal to zero. Still, the common value of its coefficients of upper and lower tail dependence is positive and depends on the degrees-of-freedom parameter. In their numerical experiments, \citet{belalia+b+l+t:2017} found that for such alternatives, the power of the Cram\'er--von Mises test based on both the empirical copula and the Bernstein empirical copula is particularly weak. Their test statistics based on the Bernstein copula density estimator performed much better.

To increase the power of the Cram\'er--von Mises statistic against such difficult alternatives, a natural approach is to follow \citet{dewet:1980} and introduce a weight function emphasizing the tails. For $\gamma \in [0, 2)$, we propose the weighted Cram\'er--von Mises statistic
\begin{equation}
\label{eq:CvM}
  T_{n,\gamma} 
  = 
  n \int_{[0,1]^d} \frac{\{ C_n^\beta(\vect u) - C(\vect u) \}^2}{\{g(\vect u)\}^\gamma} \, \mathrm{d} \vect u. 
\end{equation}
We are mostly interested in the case where $C(\vect u) = \Pi(\vect u) = \prod_{j=1}^d u_j$, the independence copula. If $\gamma = 0$, the weight function disappears and we are back to the original Cram\'er--von Mises statistic, but with the empirical beta copula replacing the empirical copula. 

\begin{corollary}
Under the assumptions of Theorem~\ref{thm:main}, we have, for every $\gamma \in [0, 2)$, the weak convergence
\[
  T_{n,\gamma} \dto T_\gamma = \int_{[0, 1]^d} \frac{\{\Cb_C(\vect u)\}^2}{\{g(\vect u)\}^\gamma} \, \mathrm{d} \vect{u},
  \qquad n \to \infty.
\]
This is particularly true in case of independent random sampling from a $d$-variate distribution with continuous margins and independent components ($C = \Pi$).
\end{corollary}

\begin{proof}
We have
\[
  T_{n, \gamma}
  =
  \int_{[0,1]^d} 
    \left( \frac{\Cb_n^\beta(\vect u)}{\{g(\vect u)\}^{\gamma/4}}\right)^2 \,
    \frac{1}{\{g(\vect u)\}^{\gamma/2}} \,
  \mathrm{d} \vect u.
\]
By Theorem~\ref{thm:main} applied to $\omega = \gamma / 4 \in [0, 1/2)$, the first part of the integrand converges weakly, in $\ell^\infty([0, 1]^d)$, to the stochastic process $(\Cb_C/g^{\gamma/4})^2$. Further, since $\gamma/2 \in [0, 1)$, the integral $\int_{[0, 1]^d} \{ g(\vect u) \}^{-\gamma/2} \, \mathrm{d} \vect u$ is finite. The linear functional that sends a measurable function $f \in \ell^\infty([0, 1]^d)$ to the scalar $\int_{[0, 1]^d} f( \vect u ) \, \{ g(\vect u) \}^{-\gamma/2} \, \mathrm{d} \vect u$ is therefore bounded. The conclusion follows from the continuous mapping theorem.
\end{proof}

A comprehensive simulation study comparing the performance of the weighted Cram\'er--von Mises statistic against all competitors and for a wide range of tuning parameters and data-generating processes is out of this paper's scope. We limit ourselves to the case identified as the most difficult one in \citet{belalia+b+l+t:2017}, the bivariate \emph{t} copula with zero correlation parameter. We copy the settings in their Section~5: the degrees-of-freedom parameter is $\nu = 2$ and we consider independent random samples of size $n \in \{100, 200, 400, 500\}$. We compare the power of our statistic $T_{n,\gamma}$ with the powers of their statistics $T_n,\delta_n,I_n$ at the $\alpha = 5\%$ significance level based on $1\,000$ replications. 

We implemented our estimator in the statistical software environment \textsf{R} \citep{Rlanguage} using the package \textsf{copula} \citep{KojYan10R}. The critical values were computed by a Monte Carlo approximation based on $10\,000$ random samples from the uniform distribution on the unit square. For the statistics in \cite{belalia+b+l+t:2017}, we copied the relevant values from their Tables~4, 5, and~6. Their statistics depend on the degree, $k$, of the Bernstein polynomials, which they selected in $\{5, 10, \ldots, 30\}$. Note that for $\gamma = 0$ and $k = n$, our statistic $T_{n,\gamma}$ coincides with their statistic $T_n$. Their statistics $\delta_n$ and $I_n$ are based on the Bernstein copula density estimator in \cite{bouezmarni+r+t:2010}.

The results are presented in Table~\ref{tab:power}. The unweighted Cram\'er--von Mises statistic $T_n$ does a poor job in detecting the alternative. The novel statistics $\delta_n$ and $I_n$ in \cite{belalia+b+l+t:2017} are more powerful, especially the statistic $I_n$, which is a Cram\'er--von Mises statistic based on the Bernstein copula density estimator. For the weighted Cram\'er--von Mises statistic $T_{n,\gamma}$, the power increases with $\gamma$. For the largest considered value, $\gamma = 1.75$, the power is higher than the one of $T_n$, $\delta_n$ and $I_n$ for any value of $k$ considered.

\begin{table}
\begin{center}
\begin{tabular}{r|r@{\qquad}rrr|r@{\qquad}r}
\toprule
&$k$& $T_n$ & $\delta_n$ & $I_n$ & $\gamma$ & $T_{n,\gamma}$ \\
\midrule\midrule
$n=100$ & $\phantom{0}5$ & 0.056 & 0.114 & 0.094 & $0.25$  & 0.102\\
& $10$ & 0.064 & 0.130 & 0.168 & $0.50$ & 0.091\\
& $15$ & 0.066 & 0.166 & 0.254 & $0.75$ & 0.138\\
& $20$ & 0.070 & 0.132 & 0.270 & $1.00$ & 0.179\\
& $25$ & 0.070 & 0.102 & 0.284 & $1.25$ & 0.216\\
& $30$ & 0.068 & 0.114 & 0.294 & $1.50$ & 0.292\\
& & & & & $1.75$ & 0.401\\
\midrule
$n=200$ & $\phantom{0}5$ & 0.076 & 0.176 & 0.094 & $0.25$ & 0.123\\
& $10$ & 0.080 & 0.222 & 0.308 & $0.50$ & 0.161\\
& $15$ & 0.088 & 0.226 & 0.442 & $0.75$ & 0.233\\
& $20$ & 0.094 & 0.210 & 0.466 & $1.00$ & 0.335\\
& $25$ & 0.096 & 0.176 & 0.472 & $1.25$ & 0.428\\
& $30$ & 0.086 & 0.148 & 0.458 & $1.50$ & 0.605\\
& & & & & $1.75$ & 0.705\\
\midrule
$n=400$ & $\phantom{0}5$ & 0.044 & 0.366 & 0.230 & $0.25$ & 0.278\\
& $10$ & 0.038 & 0.492 & 0.588 & $0.50$ & 0.427\\
& $15$ & 0.048 & 0.472 & 0.702 & $0.75$ & 0.555\\
& $20$ & 0.044 & 0.432 & 0.762 & $1.00$ & 0.777\\
& $25$ & 0.048 & 0.382 & 0.772 & $1.25$ & 0.864\\
& $30$ & 0.050 & 0.354 & 0.780 & $1.50$ & 0.930\\
& & & & & $1.75$ & 0.964\\
\midrule
$n=500$ & $\phantom{0}5$ & 0.072 & 0.398 & 0.192 & $0.25$ &0.406 \\
& $10$ & 0.096 & 0.542 & 0.688 & $0.50$ & 0.588\\
& $15$ & 0.100 & 0.552 & 0.746 & $0.75$ & 0.773\\
& $20$ & 0.110 & 0.506 & 0.806 & $1.00$ & 0.883\\
& $25$ & 0.106 & 0.476 & 0.824 & $1.25$ & 0.966\\
& $30$ & 0.096 & 0.458 & 0.824 & $1.50$ & 0.986\\
& & & & & $1.75$ & 0.992\\
\bottomrule
\end{tabular}
\end{center}
\caption{\label{tab:power}Testing the independence hypothesis when the true copula is equal to the \emph{t} copula with zero correlation parameter and degrees-of-freedom parameter $\nu = 2$. Powers based on $1\,000$ random samples of sizes $n \in \{100, 200, 400, 500\}$ at significance level $\alpha = 5\%$. Comparison between, on the one hand, the statistics $T_n,\delta_n,I_n$ in \citet{belalia+b+l+t:2017} with degree $k$ of the Bernstein polynomials and, on the other hand, the weighted Cram\'er--von Mises statistic $T_{n,\gamma}$ in Eq.~\eqref{eq:CvM} with weight parameter $\gamma$. The values in the columns headed $T_n$, $\delta_n$ and $I_n$ have been copied from Tables~4--6 in \cite{belalia+b+l+t:2017}.}
\end{table}

\section{Application: nonparametric estimation of a Pickands dependence function}
\label{sec:pick}

A $d$-variate copula $C$ is a multivariate extreme-value copula if and only if it can be written as
\[ 
  C(\vect u) 
  = 
  \exp \left\{ 
    \left( \sum_{j=1}^d \log u_j \right) \, 
    A  \left( \frac{\log u_1}{\sum_{j=1}^d \log u_j}, \dots, \frac{\log u_{d-1}}{\sum_{j=1}^d \log u_j} \right)
  \right\},
\]
for $\vect u \in (0,1]^d \setminus \{ (1, \ldots, 1) \}$. The function $A:\Delta_{d-1} \to [1/d,1]$ is called the Pickands dependence function \citep[after][]{Pic81}, its domain being the unit simplex $\Delta_{d-1} = \{ \vect t=(t_1, \dots , t_{d-1}) \in [0,1]^{d-1}: \sum_{j=1}^{d-1}t_j\leq 1\}$.  

Writing $t_d = t_d(\vect t) = 1 - t_1 - \cdots - t_{d-1}$ for $\vect t \in \Delta_{d-1}$, we have $C(u^{t_1}, \ldots, u^{t_d}) = u^{A(\vect t)}$ for $0 < u < 1$, and thus
\[
  \log \{ A( \vect t ) \}
  =
  - \gamma + \int_0^1 \left\{ C(u^{t_1}, \ldots, u^{t_d}) - \ind_{[e^{-1}, 1]}(u) \right\} \, \frac{\diff u}{u \log u},  
\]
where $\gamma = 0.5772156649\ldots$ is the Euler--Mascheroni constant. The rank-based Cap\'er\`aa--Foug\`eres--Genest (CFG) estimator, $\Acfg(\vect t)$, arises by replacing $C$ in the above formula by the empirical copula, $\Cn$; see \citep{CapFouGen97} for the original estimator and see \citep{GenSeg09, gudendorf+s:2012} for the rank-based versions in dimensions two and higher, respectively. We now propose to replace $C$ by the empirical beta copula \eqref{eq:empBetaCop} instead, which gives the estimator
\begin{equation}
\label{eq:CFG:b}
  \log \{ \Acfgb(\vect t) \}
  =
  - \gamma + \int_0^1 \left\{ \Cnb(u^{t_1}, \ldots, u^{t_d}) - \ind_{[e^{-1}, 1]}(u) \right\} \, \frac{\diff u}{u \log u}.
\end{equation}
The technique could also be used for other estimators based upon the empirical copula \citep{BucDetVol11,BerBucDet13}. 

For the CFG-estimator on usually employs the endpoint-corrected version
\begin{equation}
\label{eq:CFG:c}
  \log \{ \Acfgc(\vect t) \}
  =
  \log \{ \Acfg(\vect t) \}
  -
  \sum_{j=1}^d t_j \log \{ \Acfg( \vect e_j ) \},
\end{equation}
where $\vect e_j = (0, \ldots, 0, 1, 0, \ldots, 0)$ is the $j$-th canonical unit vector in $\mathbb{R}^d$. For the estimator based on the empirical beta copula the endpoint correction is immaterial, since $\Cnb$ is a copula itself and thus $\log \Acfgb( \vect e_j ) = 0$ for all $j = 1, \ldots, d$.

Thanks to Theorem~\ref{thm:main}, the limit of the beta CFG estimator can be derived from Theorem~\ref{thm:main} by a straightforward application of the continuous mapping theorem. The result does not require serial independence and can be extended to higher dimensions.

\begin{corollary}
Let $C$ be a $d$-variate extreme-value copula with Pickands dependence function $A : \Delta_{d-1} \to [1/d, 1]$. Under the assumptions of Theorem~\ref{thm:main} we have, as $n \to \infty$,
\[
  \sqrt n \left\{ \Acfgb(\point) - A(\point) \right\} 
  \dto \mathbb{A}(\point) \; \text{ in } \ell^\infty(\Delta_{d-1}),
\]
where, for $\vect t \in \Delta_{d-1}$, we define $\mathbb{A}(\vect t) = A( \vect t ) \int_0^1 \mathbb{C}_C(u^{t_1}, \ldots, u^{t_d}) \, (u \log u)^{-1} \, \mathrm{d}u$.
\end{corollary}

\begin{proof}
Let $0 < \omega < 1/2$. We have
\begin{align*}
  \sqrt n \left[ \log \{ \Acfgb(\vect t) \} - \log \{ A(\vect t) \} \right] 
  = 
  \int_0^1 
    \Cb_n^{\beta}(u^{t_1}, \ldots, u^{t_d}) 
  \frac{\mathrm{d}u}{u \log u} 
  = 
  \int_0^1 
    \frac{\Cb_n^{\beta}(u^{t_1}, \ldots, u^{t_d})}{\{g(u^{t_1}, \ldots, u^{t_d})\}^\omega} \, 
    \{g(u^{t_1}, \ldots, u^{t_d})\}^\omega \, 
  \frac{\mathrm{d}u}{u \log u}.
\end{align*}
The integral
$
  \int_0^1 \{g(u^{t_1}, \ldots, u^{t_d})\}^\omega \, (u \log u)^{-1} \, \mathrm{d}u
$
is bounded, uniformly in $\vect t \in \Delta_{d-1}$. Therefore, the linear map that sends a measurable function $f \in \ell^\infty([0, 1]^d)$ to the bounded function $\vect t \mapsto \int_0^1 f(u^{t_1}, \ldots, u^{t_d}) \, \{g(u^{t_1}, \ldots, u^{t_d})\}^\omega \, (u \log u)^{-1} \, \mathrm{d} u$ is continuous. By Theorem~\ref{thm:main} and the continuous mapping theorem, we find, as $n \to \infty$,
\[
  \sqrt n \left[\log \{\Acfgb(\point)\} - \log\{A(\point)\}\right]  
  \dto \left(\int_0^1 \Cb_C(u^{t_1}, \ldots, u^{t_d}) \, \frac{\mathrm{d}u}{u \log u}\right)_{\vect t \in \Delta_{d-1}}
  \; \text{ in } \ell^\infty(\Delta_{d-1}).  
\]
Finally, the result follows by an application of the functional delta method.
\end{proof}

We compare the finite-sample performance of the endpoint-corrected CFG estimator with the variant based on the empirical beta copula. As performance criterion for an estimator $\hat{A}$, we use the integrated mean squared error,
\[
  \int_{\Delta_{d-1}} \Exp \left[ \left\{ \hat{A}(\vect t) - A(\vect t) \right\}^2 \right] \, \diff \vect t
  =
  \Exp \left[ \left\{ \hat{A}(\vect T) - A(\vect T) \right\}^2 \right],
\]
where the random variable $\vect T$ is uniformly distributed on $\Delta_{d-1}$ and is independent of the sample from which $\hat{A}$ was computed. We approximate the integrated mean squared error through a Monte Carlo procedure: for a large integer $M$, we generate $M$ random samples of size $n$ from a given copula and we calculate
\[
  \frac{1}{M} \sum_{m=1}^M \left\{ \hat{A}_n^{(m)}(\vect T^{(m)}) - A(\vect T^{(m)}) \right\}^2
\]
where $\hat{A}_n^{(m)}$ denotes the estimator based upon sample number $m$, and where the random variables $\vect T^{(1)}, \ldots, \vect T^{(m)}$ are uniformly distributed on $\Delta_{d-1}$ and are independent of each other and of the copula samples. The approximation error is $\Oh_{\Prob}(1/\sqrt{M})$, aggregating both the sampling error and the integration error. A similar trick was used in \cite{SegSibTsu17} and is more efficient then first estimating the pointwise mean squared error through a Monte Carlo procedure and then integrating this out via numerical integration.

We considered the following data-generating processes:
\begin{itemize}
\item[(M1)]
independent random sampling from the bivariate Gumbel copula \cite{gumbel:1961}, which has Pickands dependence function $A(t) = \{t^{1/\alpha} + (1-t)^{1/\alpha}\}^{\alpha}$ for $t \in [0, 1]$ and with parameter $\alpha \in [0, 1]$, for which Kendall's tau is $\tau = 1-\alpha$. We also considered independent random samples from the bivariate Galambos, H\"usler--Reiss and t-EV copula families, yielding similar results as for the bivariate Gumbel copula, not shown to save space. See, e.g., \cite{GudSeg2010} for the definitions of these copulas;  

\item[(M2)]
independent random sampling from a special case of the trivariate asymmetric logistic extreme-value copula \cite{tawn:1990}, with Pickands dependence function $A(t_1, t_2) = \sum_{(i,j) \in \{(1, 2), (2, 3), (3, 1)\}} \{(\theta t_i)^{1/\alpha} + (\phi t_j)^{1/\alpha} \}^\alpha + 1 - \theta - \phi$ for $(t_1, t_2) \in \Delta_2$ and $t_3 = 1-t_1-t_2$. As in \citep[Section~5]{GudSeg12}, we set $\phi = 0.3$ and $\theta = 0.6$, and $\alpha$ varies between $0$ and $1$;

\item[(M3)]
sampling from the strictly stationary bivariate moving maximum process $(U_{t1}, U_{t2})_{t \in \mathbb{Z}}$ given by
\[
  U_{t1} = \max \left\{ W_{t-1,1}^{1/a}, W_{t1}^{1/(1-a)} \right\}
  \qquad
  \text{and}
  \qquad
  U_{t2} = \max \left\{ W_{t-1,2}^{1/b}, W_{t2}^{1/(1-b)} \right\},
\]
where $a, b \in [0, 1]$ are two parameters and where $(W_{t1}, W_{t2})_{t \in \mathbb{Z}}$ is an iid sequence of bivariate random vectors whose common distribution is an extreme value-copula with some Pickands dependence function $B$. By Eq.~(8.1) in \cite{bucher+s:2014}, the stationary distribution of $(U_{t1}, U_{t2})$ is an extreme-value copula too, and its Pickands dependence function can be easily calculated to be
\[
  A(t)
  =
  \{ a(1-t)+bt \} \, B \left( \frac{bt}{a(1-t)+bt} \right) 
  +
  \{ (1-a)(1-t) + (1-b)t \} \, B \left( \frac{(1-b)t}{(1-a)(1-t) + (1-b)t} \right),
\]
for $t \in [0, 1]$. We let $B(t) = \{t^{1/\alpha} + (1-t)^{1/\alpha}\}^{\alpha}$ for $t \in [0, 1]$ and $\alpha \in [0, 1]$ (the bivariate Gumbel copula as above, with Kendall's tau $\tau = 1-\alpha$) and we set $a = 0.1$ and $b = 0.7$, so that $A$ is asymmetric. 
\end{itemize}

The results are shown in Figure~\ref{fig:Pickands}. Each plot is based on $10\,000$ samples of size $n \in \{20, 50, 100\}$. For weak dependence (small $\tau$, large $\alpha$), the beta variant \eqref{eq:CFG:b} is the more efficient one, whereas for strong dependence (large $\tau$, small $\alpha$), it is the usual CFG estimator \eqref{eq:CFG:c} which is more accurate.

In order to gain a better understanding, we have also traced some trajectories of estimated Pickands dependence functions for independent random samples of the bivariate Gumbel copula at $\tau \in \{0.3, 0.9\}$ and $n \in \{20, 50, 100\}$; see Figure~\ref{fig:Pickands:A}. For each trajectory of the CFG estimator, there is a corresponding trajectory of the new estimator that is based on the same sample. For large $\tau$, the true extreme-value copula $C$ is close to the Fr\'echet--Hoeffding upper bound, $M(u_1, u_2) = \max(u_1, u_2)$. As a result, $C$ is strongly curved around the main diagonal $u_1 = u_2$, and this implies a strong curvature of the Pickands dependence function $A$ around $t = 1/2$. The empirical beta copula can be seen as a smoothed version of the empirical copula with an implicit bandwidth of the order $1/\sqrt{n}$ \citep[p.~47]{SegSibTsu17}. For smaller $n$, oversmoothing occurs, producing a negative bias for the empirical beta copula around the diagonal $u_1 = u_2$ and thus a positive bias for the beta variant of the CFG estimator around $t = 1/2$.

\begin{figure}
\begin{center}
\begin{tabular}{@{}ccc}
\includegraphics[width=0.32\textwidth]{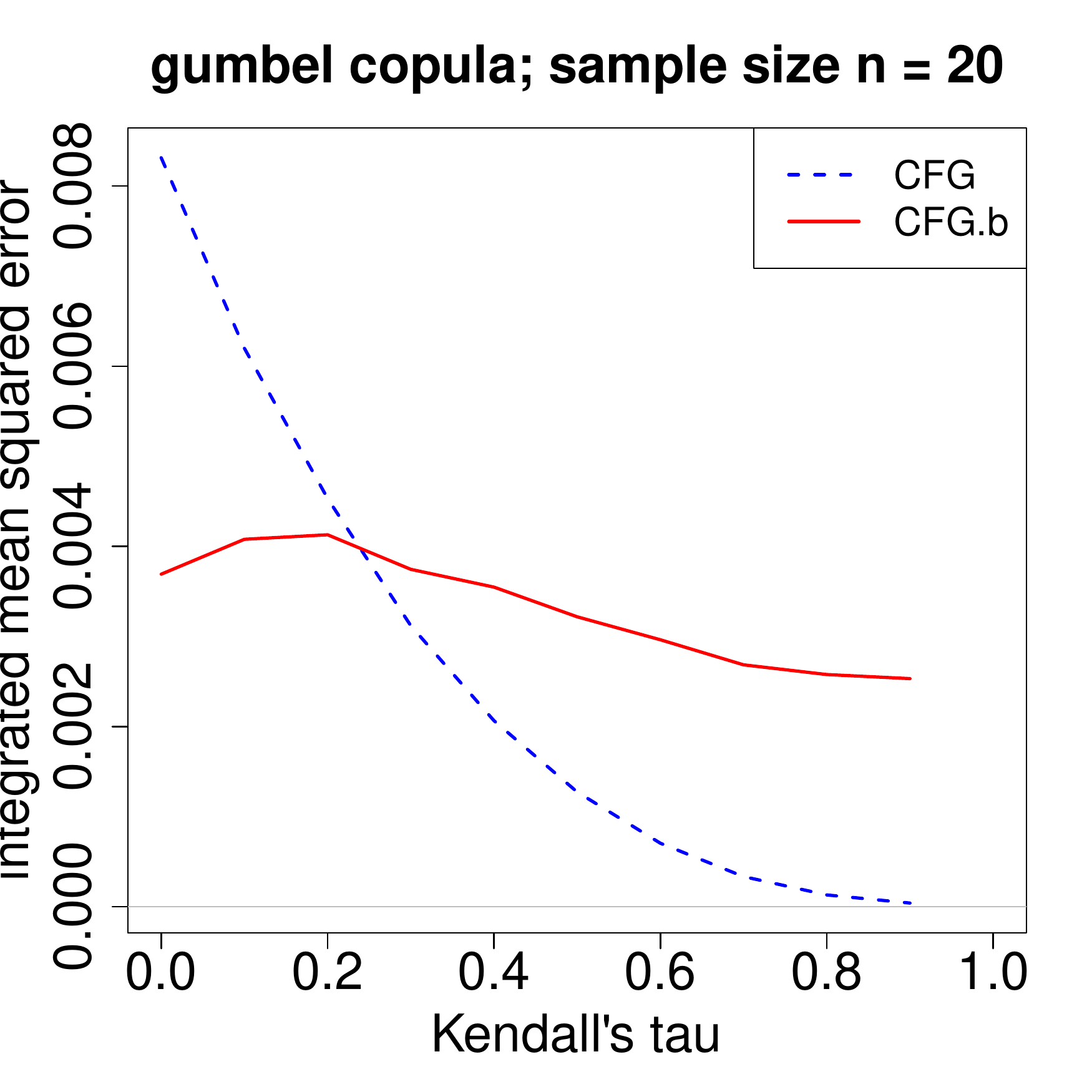}&
\includegraphics[width=0.32\textwidth]{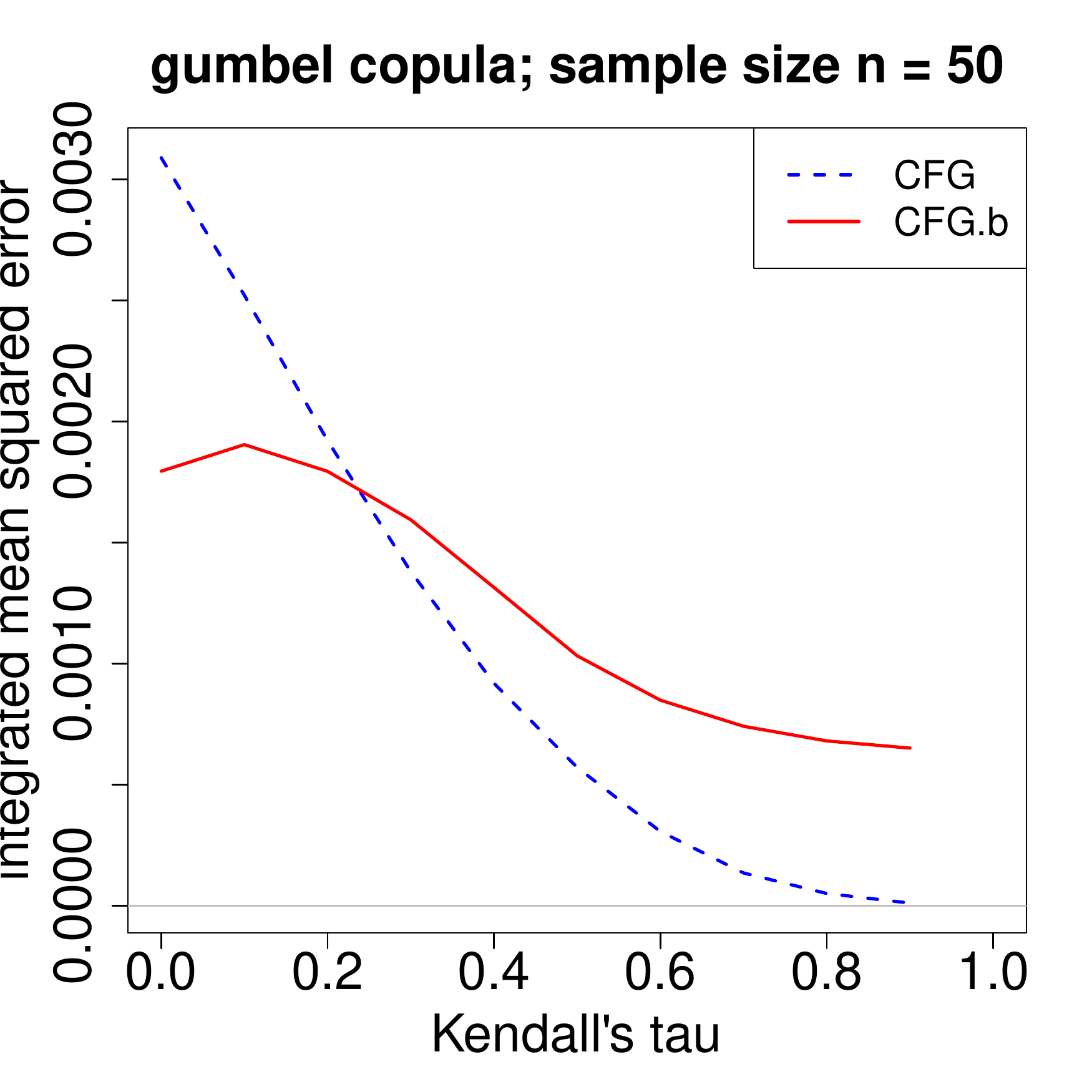}&
\includegraphics[width=0.32\textwidth]{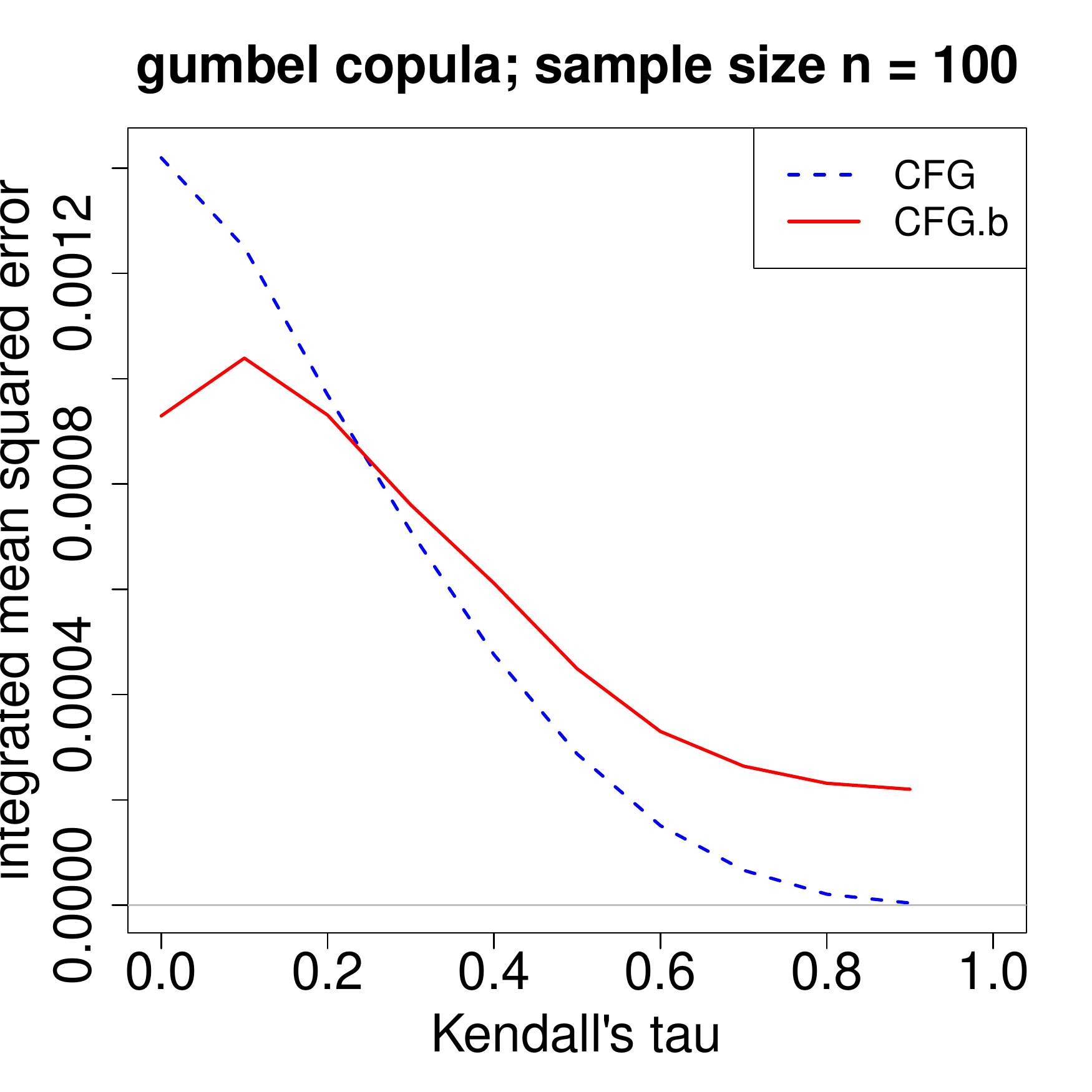}\\
\includegraphics[width=0.32\textwidth]{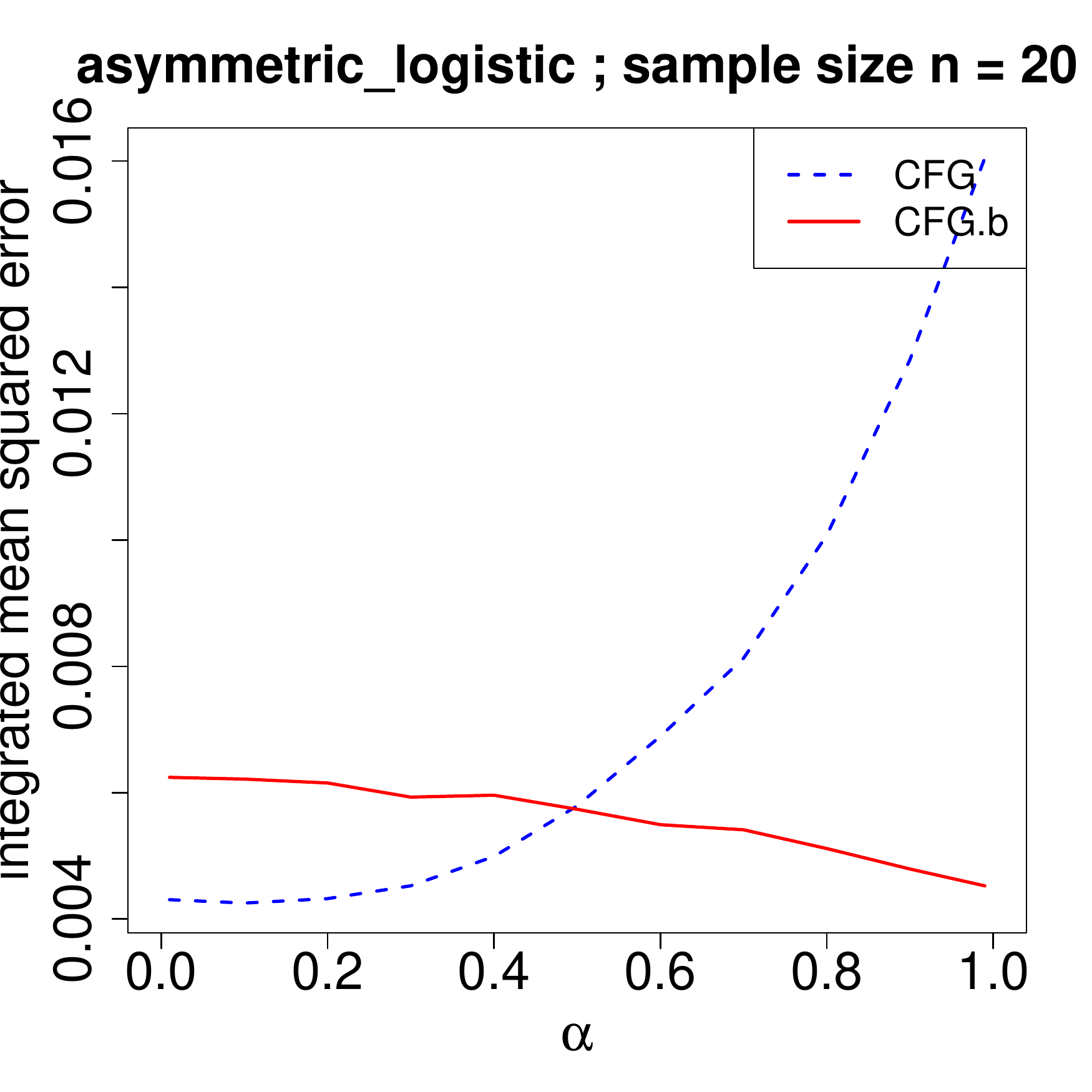}&
\includegraphics[width=0.32\textwidth]{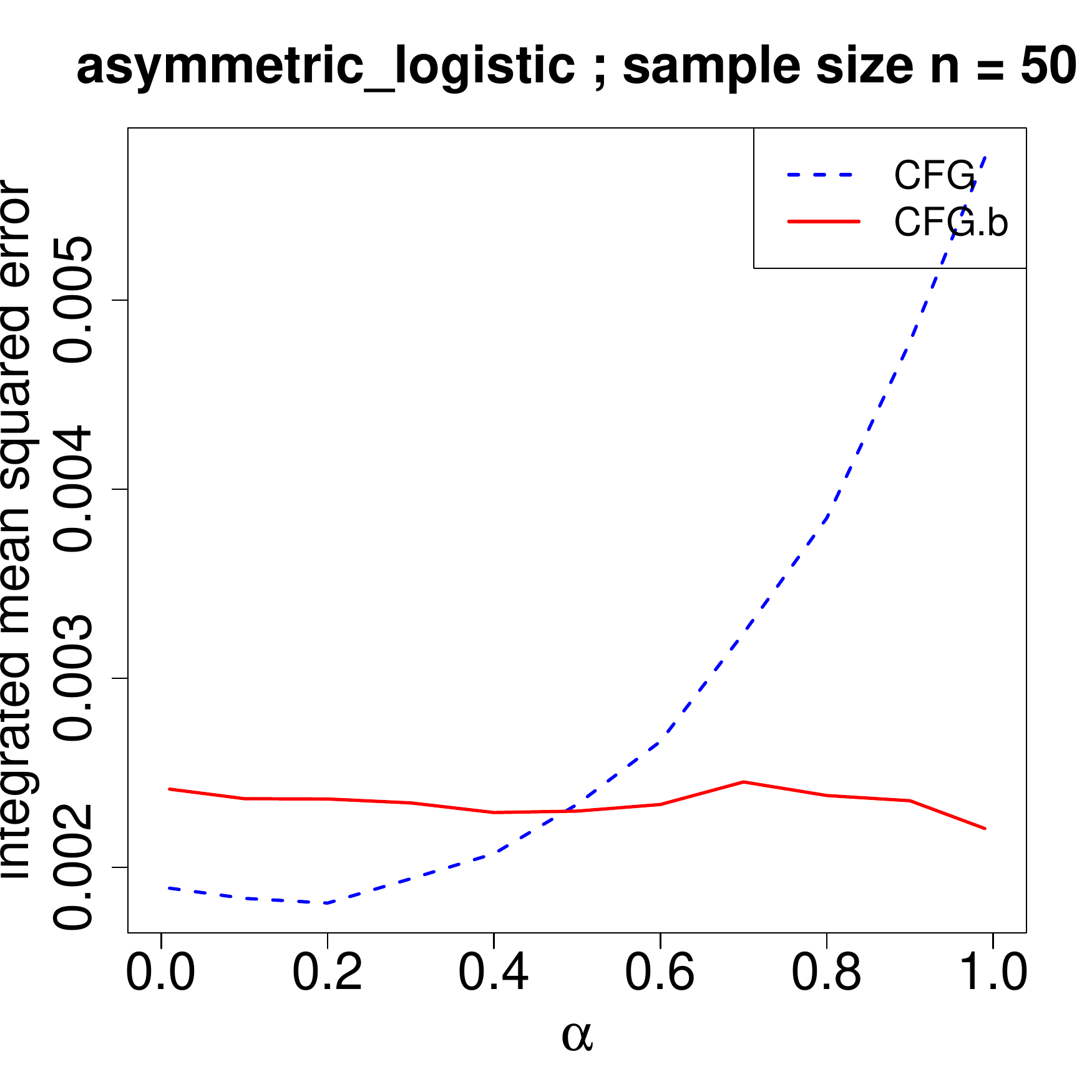}&
\includegraphics[width=0.32\textwidth]{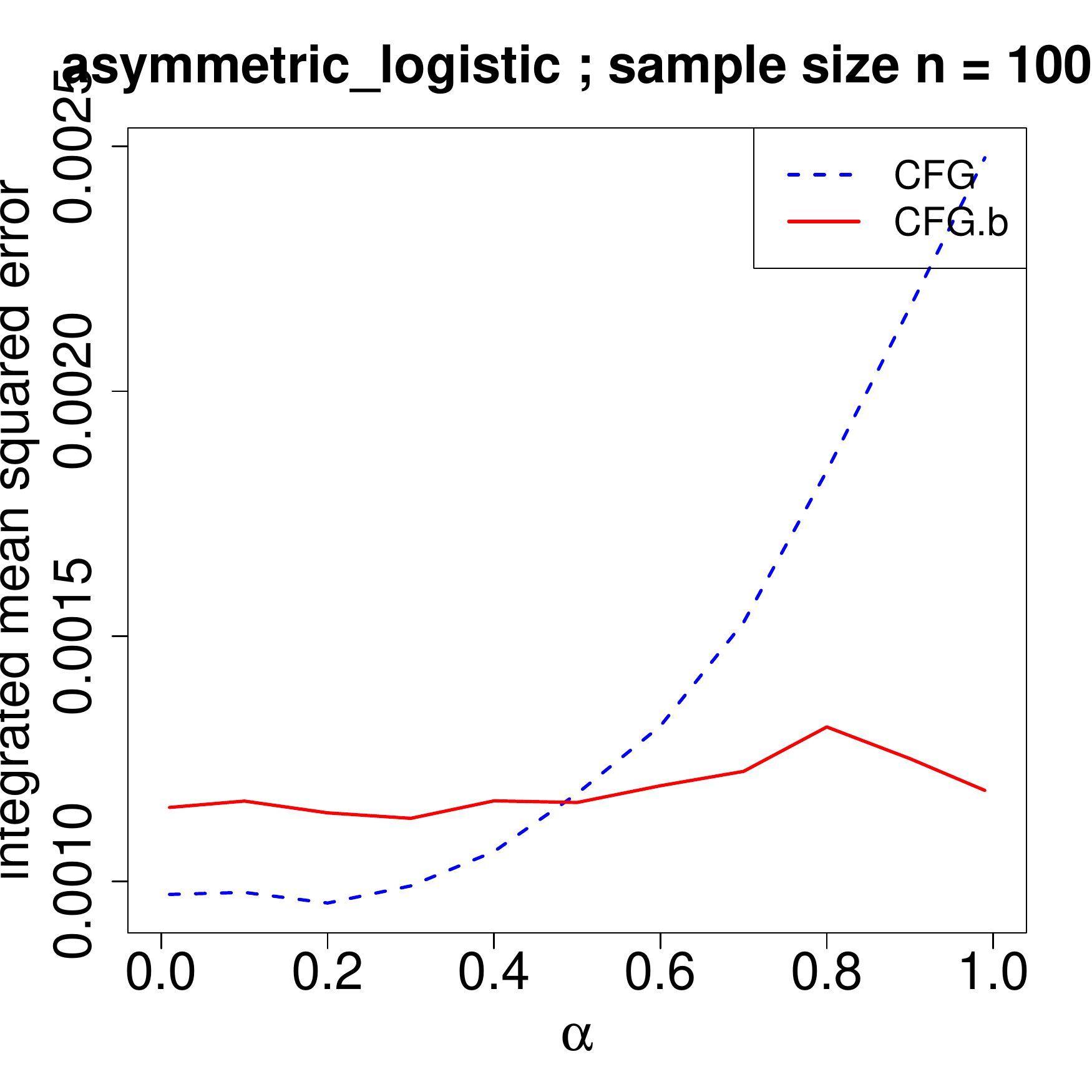}\\
\includegraphics[width=0.32\textwidth]{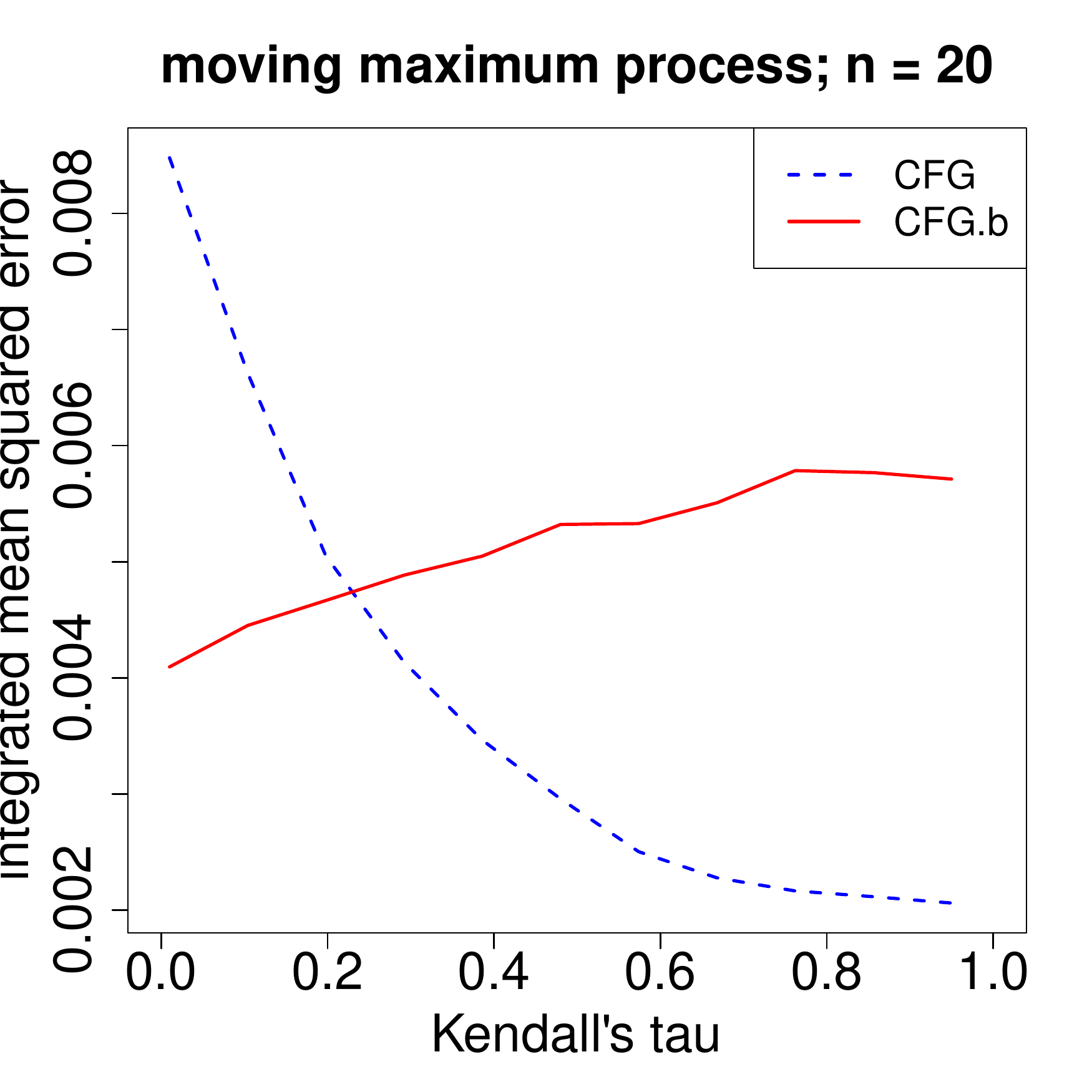}&
\includegraphics[width=0.32\textwidth]{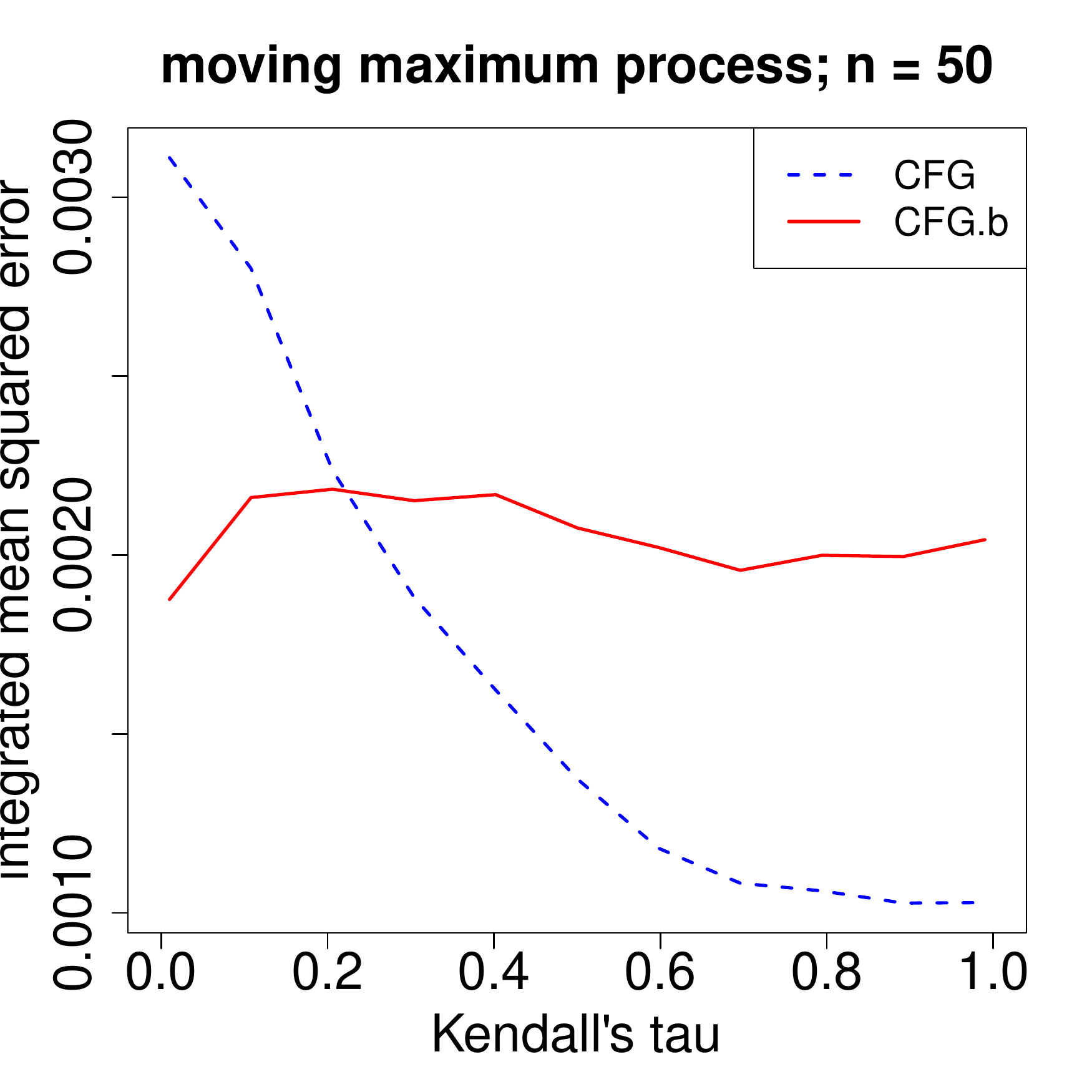}&
\includegraphics[width=0.32\textwidth]{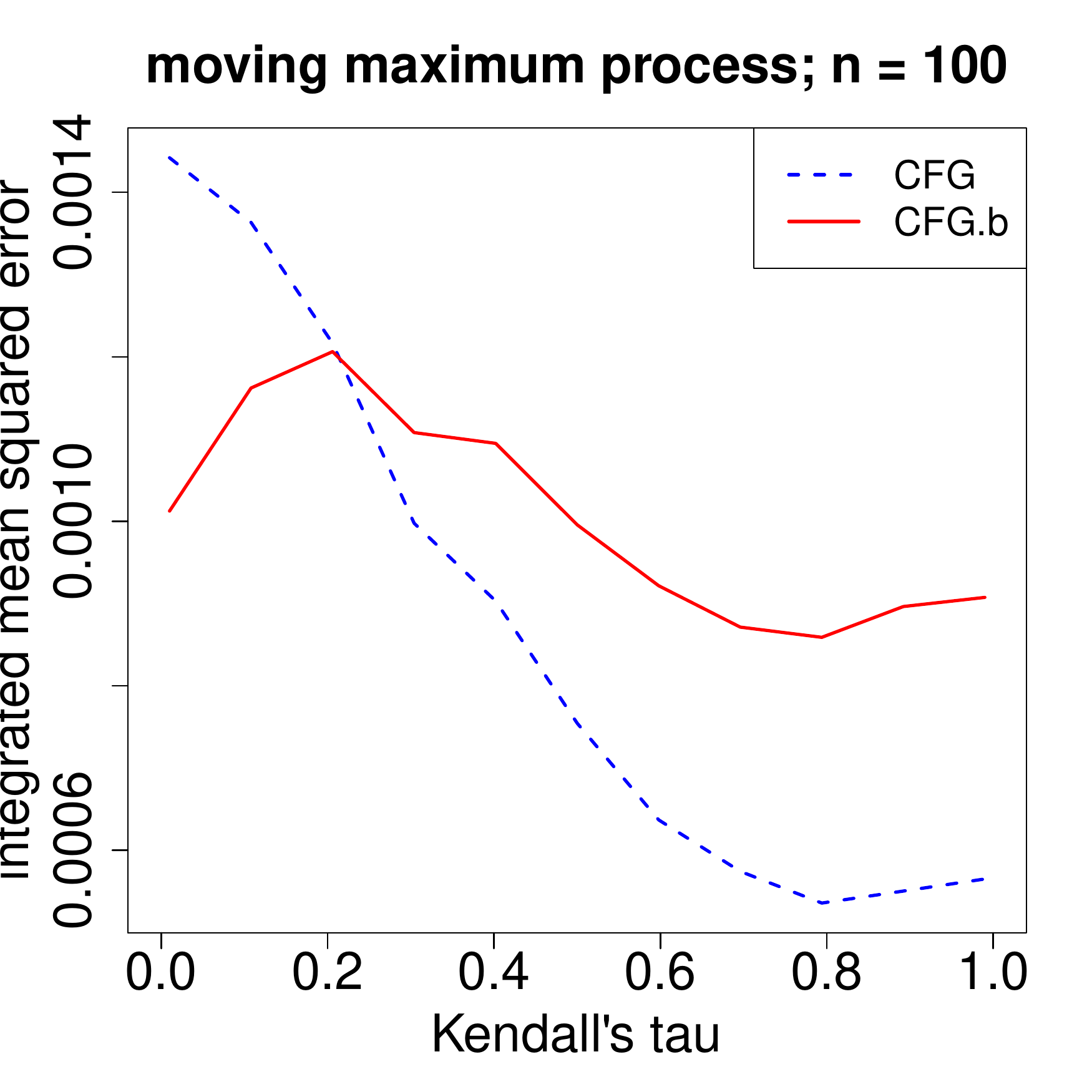}
\end{tabular}
\end{center}
\caption{\label{fig:Pickands}Integrated mean squared error (vertical axis) of the endpoint-corrected CFG-estimator \eqref{eq:CFG:c} (dashed, blue) and the empirical beta variant \eqref{eq:CFG:b} (solid, red) based on samples of the data-generating process (M1), (M2) and (M3) (top to bottom) for various choices of the parameter $\alpha$ or $\tau = 1-\alpha$. Each point is based on $10\,000$ random samples of size $n \in \{20, 50, 100\}$ (left to right).}
\end{figure}

\begin{figure}[h!]
\begin{center}
\begin{tabular}{@{}ccc}
\includegraphics[width=0.32\textwidth]{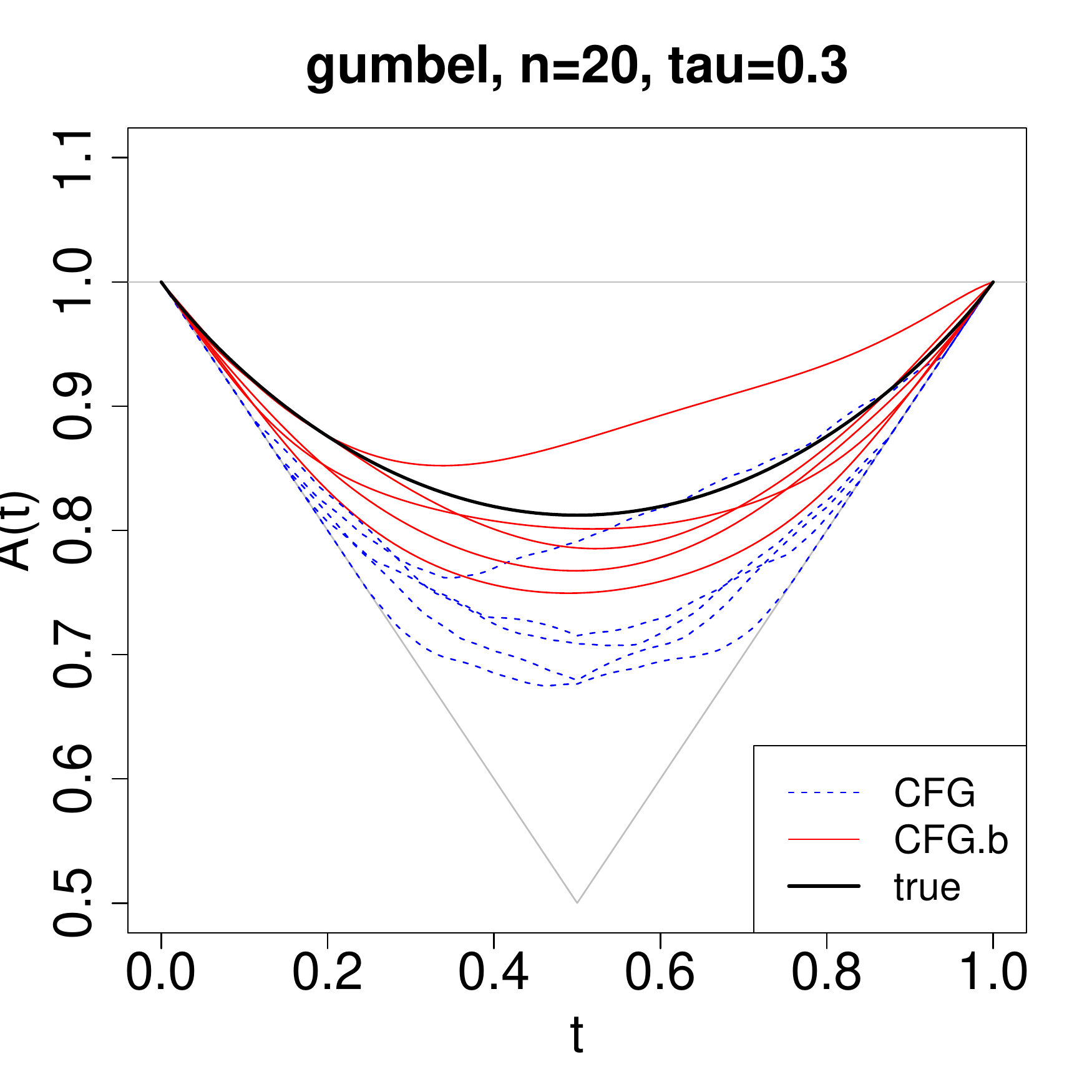}&
\includegraphics[width=0.32\textwidth]{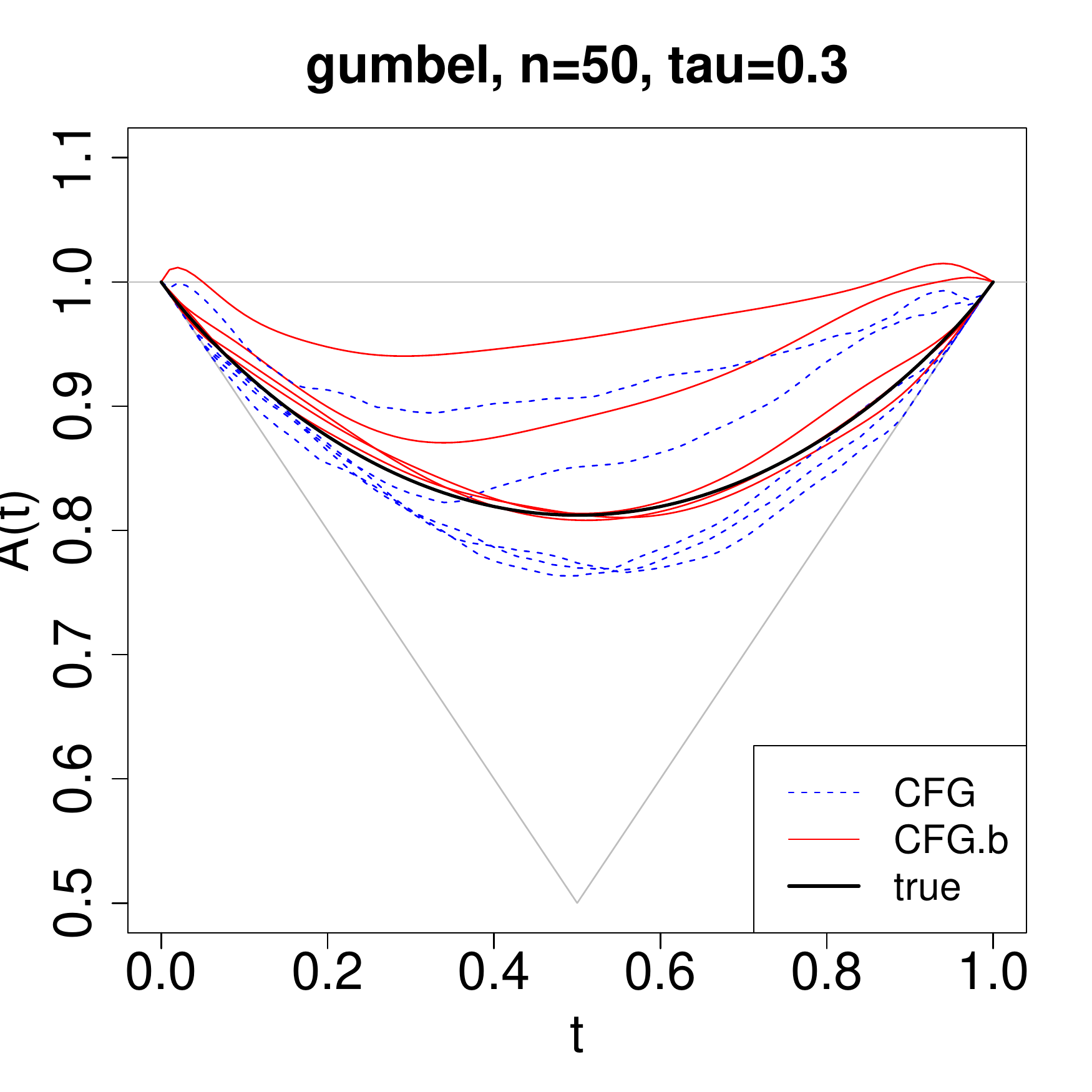}&
\includegraphics[width=0.32\textwidth]{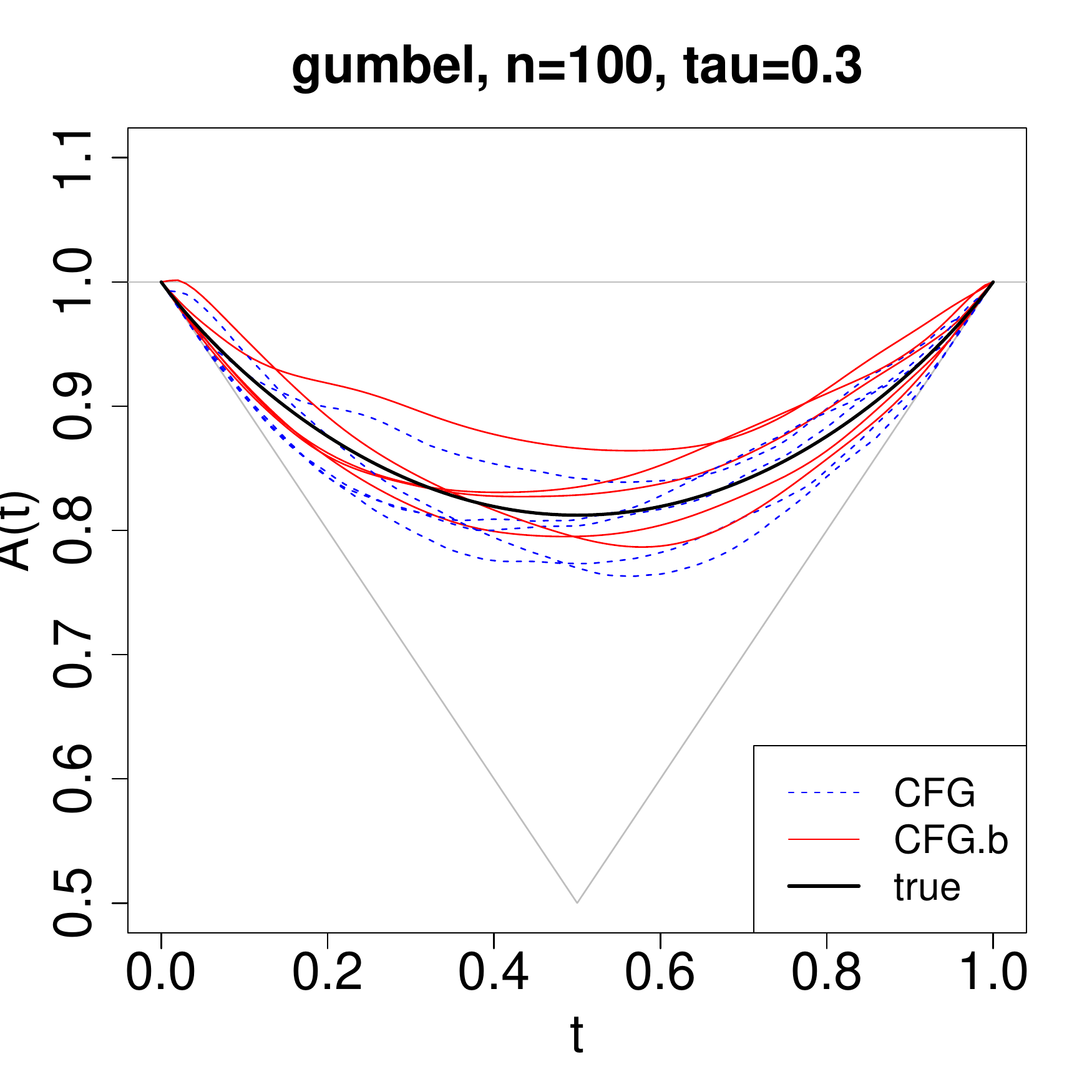}\\
\includegraphics[width=0.32\textwidth]{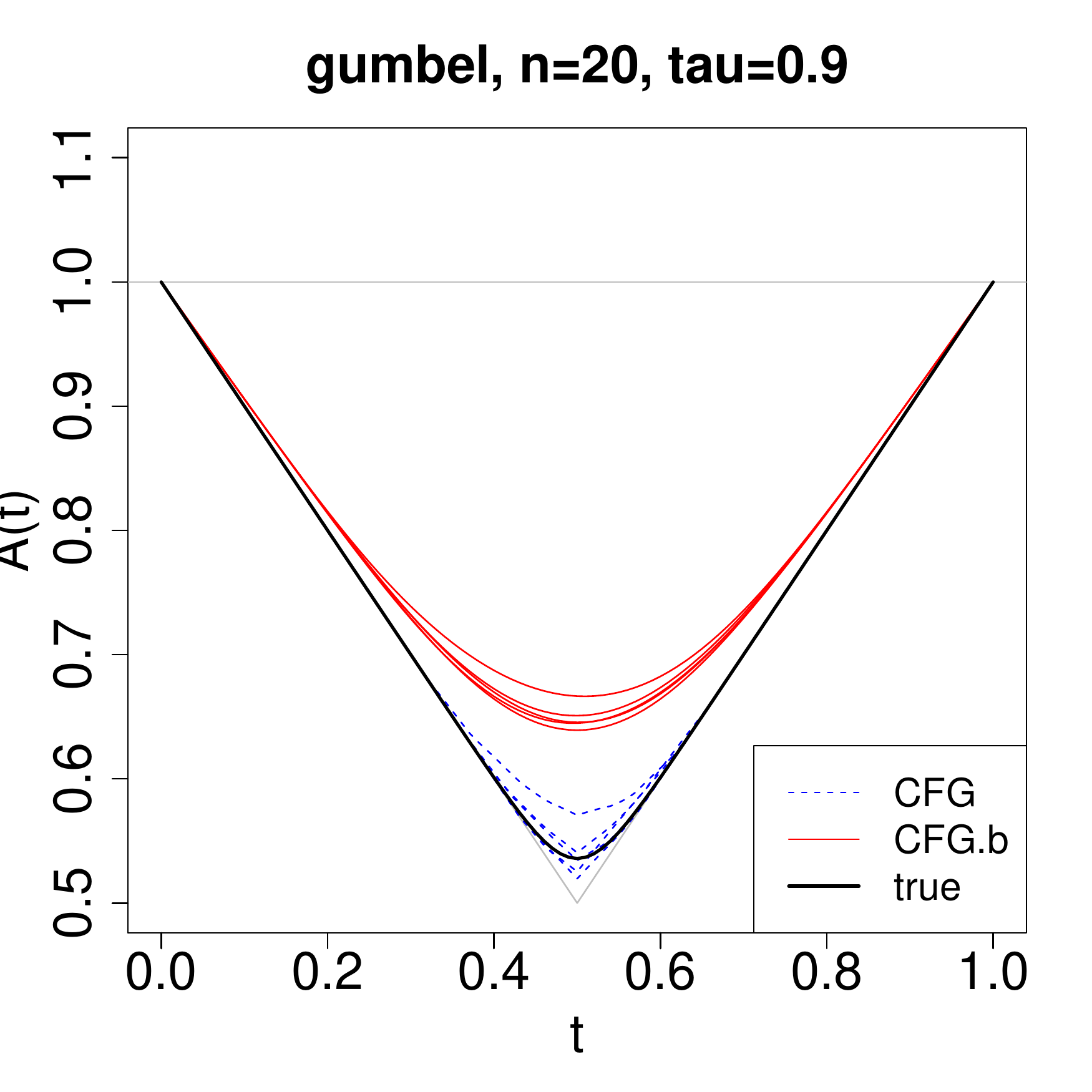}&
\includegraphics[width=0.32\textwidth]{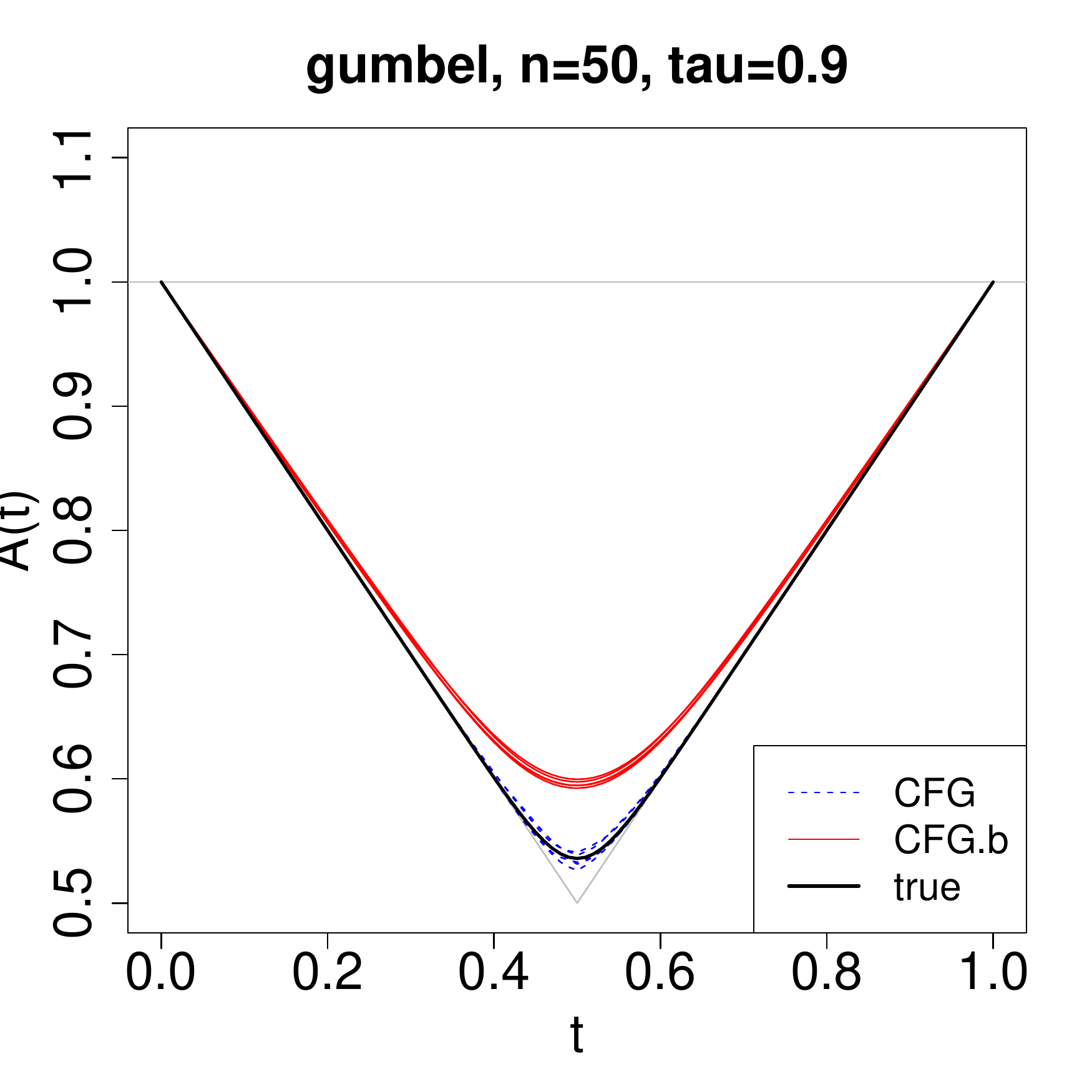}&
\includegraphics[width=0.32\textwidth]{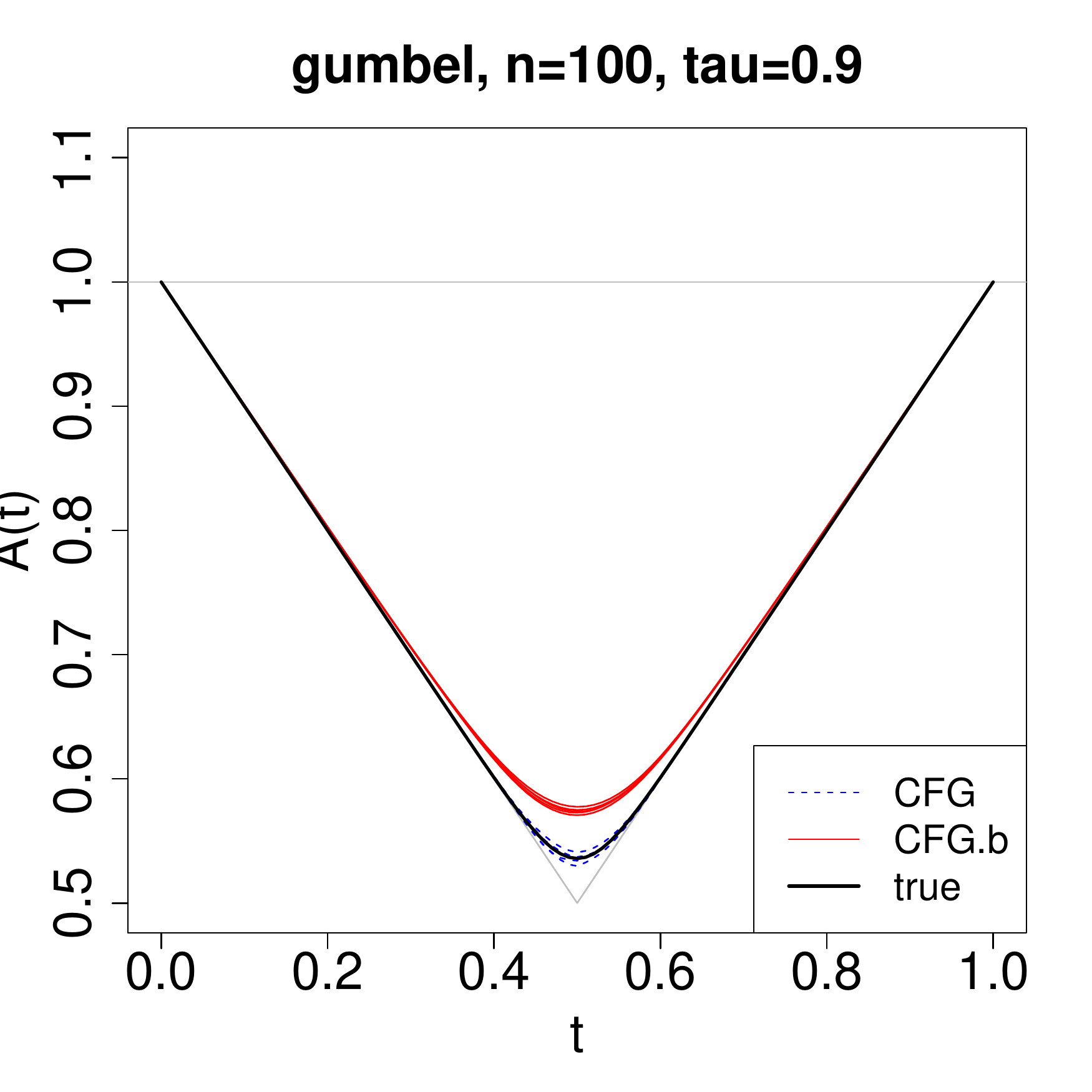}
\end{tabular}
\end{center}
\caption{\label{fig:Pickands:A}Plots of trajectories of the CFG-estimator \eqref{eq:CFG:c} (dashed blue) and the empirical beta variant \eqref{eq:CFG:b} (solid red) of the Pickands dependence function (solid black) based on samples from the Gumbel copula with Kendall's $\tau \in \{0.3, 0.9\}$ (top to bottom) and $n \in \{20, 50, 100\}$ (left to right).}
\end{figure}

\section{Proof of Theorem~\ref{thm:main}}
\label{sec:proof}

Recall the empirical copula process $\hat{\Cb}_n = \sqrt{n} (\hat{C}_n - C)$ and the empirical beta copula process $\hat{\Cb}_n^\beta = \sqrt{n} (\Cnb - C)$. The link between the empirical copula $\hat{C}_n$ and the empirical beta copula $\Cnb$ is given in \eqref{eq:empCop2empBetaCop}. In the derivation of the limit of the weighted empirical beta copula process the following decomposition plays a central role:
\begin{multline}\label{eq:decomp}
\frac {\Cb_n^\beta(\vect u )}{g(\vect u)^\omega}
 ~=~
 \frac{\hat{\Cb}_n(\vect u)}{g(\vect u )^\omega}\int_{[0,1]^d} \frac{g(\vect w)^\omega}{g(\vect u)^\omega} \diff \mu_{n,\vect u} (\vect w) \\
 ~+ ~
  \int_{[0,1]^d} \left\{ \frac{\hat {\Cb}_n(\vect w)}{g(\vect w )^\omega} - \frac{\hat{\Cb}_n(\vect u)}{g(\vect u)^\omega} \right\} \frac{g(\vect w)^\omega}{g(\vect u)^\omega} \diff \mu_{n,\vect u} (\vect w)
~ +~
\int_{[0,1]^d} \sqrt n \frac{C(\vect w) - C(\vect u)}{g(\vect u)^\omega} \diff \mu_{n,\vect u} (\vect w).
\end{multline}
It is reasonable to assume that the last two terms on the right-hand side vanish as $n\to \infty$. Indeed, the measure $\mu_{n,\vect u}$ concentrates around its mean $\vect u$, if the sample size grows, and both integrands are small if $\vect w$ is close to $\vect u$. By the same reason, the integral in the first term should be close to one. The decomposition can thus be used to obtain weak convergence of $\Cb_n^\beta / g^\omega$ on the interior of the unit cube. The boundary of the unit cube has to be treated separately. 

The case $\omega = 0$ corresponds to the unweighted case, so we assume henceforth that $0 < \omega < 1/2$. Fix a scalar $\gamma$ such that $1 / \{2(1-\omega)\} < \gamma < 1$. Consider the abbreviations $\{ g \ge n^{-\gamma} \} = \{ \vect v \in [0, 1]^d \mid g(\vect v) \ge n^{-\gamma} \}$ and similarly $\{ g < n^{-\gamma} \}$. By Lemma~\ref{lem:boundary}, we have
\begin{align*}
  \Cb_n^\beta/g^\omega
  &=
  {\Cb_n^\beta/g^\omega \ind_{\{ g \geq n^{-\gamma} \}}} + \Cb_n^\beta/g^\omega \ind_{ \{ g < n^{-\gamma} \} }  \\
  &= 
  {\Cb_n^\beta/g^\omega \ind_{\{ g \geq n^{-\gamma} \}}} + \oh(1),
  \qquad n \to \infty, \quad \text{a.s.}
\end{align*}
The three terms on the right-hand side of \eqref{eq:decomp} are treated in  Lemmas~\ref{lem:bias}, \ref{lem:int} and \ref{lem:bias2}. We find
\begin{equation}
\label{eq:CbnhatCbn}
  \Cb_n^\beta/g^\omega
  =
  \hat{\Cb}_n/g^\omega \ind_{ \{ g \geq n^{-\gamma} \} } (1 + \oh(1)) + \oh_{\Prob}(1),
  \qquad n \to \infty.
\end{equation}

Recall $\vect U_i = (U_{i,1}, \ldots, U_{i,d})$ with $U_{i,j} = F_j(X_{i,j})$. The empirical distribution function and the empirical process associated to the unobservable sample $\vect U_1, \ldots, \vect U_n$ are
\begin{align*}
  C_n(\vect u) 
  &=  \frac{1}{n} \sum_{i=1}^n \ind_{ \{\vect U_i \le \vect u\} }, &
  \alpha_n(\vect u) 
  &= \sqrt n \{ C_n(\vect u) - C(\vect u) \},
\end{align*}
respectively, for $\vect u \in [0, 1]^d$. Consider the process 
\[
  \bar \Cb_n(\vect u) 
  = 
  \alpha_n(\vect u) 
  - \sum_{j=1}^d \dot C_j(\vect u) \, \alpha_n(1,\dots,1,u_j,1,\dots,1),
  \qquad \vect u \in [0, 1]^d,
\]
with $u_j$ appearing at the $j$-th coordinate. Note the slight but convenient abuse of notation in the definition of $\bar{\Cb}_n$: if $\vect u$ is such that $u_j \in \{0, 1\}$, then $\alpha_n(1,\dots,1,u_j,1,\dots,1) = 0$ almost surely, so that the fact that for such $\vect u$, the partial derivative $\dot{C}_j(\vect u)$ has been left undefined in Condition~\ref{cond:second} plays no role.

By \eqref{eq:CbnhatCbn} above and by Theorem~2.2 in \cite{BerBucVol17} (see also Remark~\ref{rem:proof} below),
\begin{align*}
  \Cb_n^\beta/g^\omega
  &=
  \{ \bar{\Cb}_n/g^\omega + \oh_{\Prob}(1) \} \ind_{ \{ g \geq n^{-\gamma} \} } (1 + \oh(1)) + \oh_{\Prob}(1) \\
  &=
  {\bar{\Cb}_n/g^\omega \ind_{ \{ g \geq n^{-\gamma} \} }} + \oh_{\Prob}(1),
  \qquad n \to \infty.
\end{align*}
In view of Lemma~4.9 in \cite{BerBucVol17}, the indicator function can be omitted, and, applying Theorem~2.2 in the same reference again, we obtain
\[
  \Cb_n^\beta/g^\omega
  =
  \bar{\Cb}_n/g^\omega
  + \oh_{\Prob}(1)
  \dto
  \Cb_C/g^\omega,
  \qquad n \to \infty,
\]
as required. This finishes the proof of Theorem~\ref{thm:main}.

\begin{remark}
\label{rem:proof}
Some of the results in \cite{BerBucVol17} have to be adapted to the present situation.
\begin{itemize}
\item
In the latter reference, the pseudo-observations are defined as $\hat U_{i,j} = (n+1)^{-1} R_{i,j}$ rather than $n^{-1} R_{i,j}$. However, this does not affect the asymptotics, since the difference of the two empirical copulas is at most $d/n$, almost surely. For $\vect u \in \{  g \geq n^{-\gamma} \}$, this modification makes a difference of the order $\Oh_\Prob(n^{\gamma\omega+1/2-1}) = \oh_\Prob(1)$, as $n \to \infty$.
\item
In Theorem~2.2 in \cite{BerBucVol17}, the approximation of $\hat \Cb_n$ by $\bar \Cb_n$ is stated on the interior of the set $ [c/n,1-c/n]^d$ for any $c \in (0,1)$. But it can be seen in the proof of the latter statement that the result can be easily extended to the set $\{  g \geq c/n \}$. See Section~\ref{sec:BerBucVol17} below for details.
\end{itemize}
\end{remark}

\section{Auxiliary results}
\label{sec:aux}

Throughout and unless otherwise stated, we assume the conditions of Theorem~\ref{thm:main}.

\subsection{Negligibility of the boundary regions}

\begin{lemma}
\label{lem:boundary}
For $\gamma > 1 / \{2(1-\omega)\}$, we have
\[
  \sup_{\vect u \in \{g \leq n^{-\gamma}\}} \lvert \Cb_n^\beta(\vect u)/g(\vect u)^\omega \rvert = \oh(1),
  \qquad n \to \infty,
  \quad \text{a.s.}
\]
\end{lemma}

\begin{proof}
Let $\gamma > 1 / \{2(1-\omega)\}$ and $\vect u \in \{ g \leq n^{-\gamma}\}$. Without loss of generality, we only need to consider the cases $g (\vect u) = u_1$ and $g(\vect u)= 1-u_1$. The remaining cases can be treated analogously.

Let us start with the case $g(\vect u)=u_1 \leq n^{-\gamma}$. Since $\Cnb$ is a copula almost surely, we have $\Cnb (\vect u) \leq u_1$. This in turn gives us
\[
 \lvert \Cb_n^\beta(\vect u)/g(\vect u)^\omega \rvert 
 \leq 
 \sqrt n \, u_1^{-\omega} \lvert \Cnb(\vect u) + C(\vect u) \rvert 
 \leq 
  2 \sqrt n \, u_1^{1-\omega}
   \leq
    2 n^{1/2 +\gamma\omega -\gamma},
  \qquad \text{a.s.},
\]
an upper bound which vanishes as $n \to \infty$ by the choice of $\gamma$.

Now suppose that $g(\vect u)=1-u_1 \leq n^{-\gamma}$. By the definition of $g(\vect u)$, we can assume without loss of generality that $1-u_j \leq 1-u_1$ for $j=3,\dots ,d$. Again, we will use the fact that $\Cnb$ is a copula almost surely. Note that $\Cnb(1,u_2,1,\dots ,1)= u_2$. Hence, by the Lipschitz continuity of copulas we obtain, almost surely,
\begin{align*}
 \lvert \Cb_n^\beta(\vect u)/g(\vect u)^\omega \rvert 
  &\leq
  \sqrt n (1-u_1)^{-\omega}
  \{ \lvert \Cnb(\vect u) -u_2 \rvert + \lvert C(\vect u)-u_2 \rvert \} \\
  &\leq
  \textstyle{2 \sqrt n (1-u_1)^{-\omega} \sum_{j\ne 2} (1-u_j)} \\
  &\leq 
  \textstyle{2 \sqrt n \sum_{j \ne 2} (1-u_j)^{1-\omega}} \\
  &\leq
  2(d-1) n^{1/2+\gamma\omega-\gamma}
  = \oh(1),
  \qquad n \to \infty.
\end{align*}

The upper bounds do not depend on $\vect u \in \{ g \le n^{-\gamma} \}$, whence the uniformity in $\vect u$.
\end{proof}

\subsection{The three terms in the decomposition (\ref{eq:decomp})}

The following lemma is to be compared with Proposition~3.5 in \cite{SegSibTsu17}. There, a pointwise approximation rate of $\Oh(n^{-1})$ was established. Here, we state a rate which is slightly slower, $\Oh(n^{-1} \log n)$, but uniformly in $\vect u$. 

\begin{lemma}
\label{lem:bias}
If $C$ satisfies Condition \ref{cond:second}, then
\[
  \sup_{\vect u \in [0, 1]^d} 
  \left\lvert 
    \int_{[0, 1]^d} \{ C(\vect w) - C(\vect u) \} \diff \mu_{n, \vect u}(\vect w)
  \right\rvert
  =
  \Oh( n^{-1} \log n ),
  \qquad n \to \infty.
\]
\end{lemma}

\begin{proof}
Put $\eps_n = n^{-1} \log n$. First, we show that we can ignore those $\vect u$ for which $u_j \le \eps_n$ for some $j \in \{1, \ldots, d\}$. Indeed, for such $\vect u$, the absolute value in the statement is bounded by
\[
  \int_{[0, 1]^d} w_j \diff \mu_{n, \vect u}(\vect w) + u_j = 2u_j \le 2\eps_n.
\]

Let $\vect u \in [\eps_n, 1]^d$. We show how to reduce the analysis to the case where $\vect u \in [\eps_n, 1-\eps_n]^d$. Let $J = J(\vect u)$ denote the set of indices $j = 1, \ldots, d$ such that $u_j > 1-\eps_n$ and suppose that $J$ is not empty. Consider the vector $\vect e \in \{0, 1\}^d$ which has components $e_j = 1$ for $j \in J$ and $e_j = 0$ otherwise. For $\vect v \in [0, 1]^d$, the vector $\vect v \vee \vect e$ has components $(\vect v \vee \vect e)_j$ equal to $v_j$ if $j \not\in J$ and to $1$ if $j \in J$. Recall that copulas are Lipschitz continuous with respect to the $L^1$ norm with Lipschitz constant $1$. It follows that
\begin{multline*}
    \left\lvert 
      \int_{[0, 1]^d} \{ C(\vect w) - C(\vect u) \} \diff \mu_{n, \vect u}(\vect w)
    \right\rvert
  \le
  \left\lvert 
    \int_{[0, 1]^d} \{ C(\vect w \vee \vect e) - C(\vect u \vee \vect e) \} \diff \mu_{n, \vect u}(\vect w)
  \right\rvert \\
  + \int_{[0, 1]^d} \lvert C(\vect w) - C(\vect w \vee \vect e) \rvert \diff \mu_{n, \vect u}(\vect w)
  + \lvert C(\vect u \vee \vect e) - C(\vect u) \rvert.
\end{multline*}
\begin{itemize}[leftmargin=*]
\item
The first integral on the right-hand side does not depend on the variables $w_j$ for $j \in J$. It can therefore be reduced to an integral as in the statement of the lemma with respect to the variables in the set $\{1, \ldots, d\} \setminus J$. The copula of those variables is a multivariate margin of the original copula and Condition~\ref{cond:second} applies to it as well. By construction, all remaining $u_j$ are in the interval $[\eps_n, 1-\eps_n]$, as required. 
\item
We have $\lvert C(\vect w) - C(\vect w \vee \vect e) \rvert \le \sum_{j \in J} \lvert w_j - 1 \rvert \le \sum_{j \in J} (\lvert w_j - u_j \rvert + \eps_n)$. By the Cauchy--Schwarz inequality, $\int_{[0, 1]^d} \lvert w_j - u_j \rvert \diff \mu_{n, \vect u}(\vect w) \le \{n^{-1} u_j(1-u_j)\}^{1/2} \le n^{-1/2} \eps_n^{1/2} \le \eps_n$. Hence $\int_{[0,1]^d}\lvert C(\vect w) - C(\vect w \vee \vect e) \rvert  \diff \mu_{n, \vect u}(\vect w)\le 2d\eps_n$.
\item
Finally, $\lvert C(\vect u \vee \vect e) - C(\vect u) \rvert \le \sum_{j \in J} (1 - u_j) \le d \eps_n$.
\end{itemize}

It remains to consider the case $\vect u \in [\eps_n, 1-\eps_n]^d$. As in the proof of Proposition~3.4 in \cite{SegSibTsu17}, we have
\[
  \int_{[0, 1]^d} \{ C(\vect w) - C(\vect u) \} \diff \mu_{n, \vect u}(\vect w) \\
  =
  \sum_{j=1}^d 
  \int_0^1 
    \left[
      \int_{[0, 1]^d} 
	(w_j - u_j) 
	\bigl\{ \dot{C}_j(\vect u + t(\vect w - \vect u)) - \dot{C}_j(\vect u) \bigr\} 
      \diff \mu_{n, \vect u}(\vect w) 
    \right]
  \diff t.
\]
It is sufficient to show that the absolute value of the integral in square brackets is $\Oh(\eps_n)$, uniformly in $j \in \{1, \ldots, d\}$ and $t \in (0, 1)$ and $\vect u \in [\eps_n, 1-\eps_n]^d$.

The integral over $[0, 1]^d$ can be reduced to an integral over $(0, 1)^d$: indeed, the integrand is bounded in absolute value by $1$ (recall $0 \le \dot{C}_j \le 1$), and the mass on the boundary is $\mu_{n, \vect u}([0, 1]^d \setminus (0, 1)^d) = \Prob[\exists j : S_j \in \{0, n\}] \le 2d (1 - \eps_n)^n \le 2d \exp(-n\eps_n) = 2d n^{-1} = \oh(\eps_n)$ as $n \to \infty$.

In view of the second part of Condition~\ref{cond:second}, we have
\begin{multline*}
  \int_{(0, 1)^d} 
    (w_j - u_j) 
    \bigl\{ \dot{C}_j(\vect u + t(\vect w - \vect u)) - \dot{C}_j(\vect u) \bigr\} 
  \diff \mu_{n, \vect u}(\vect w) \\
  =
  t
  \sum_{k=1}^d
  \int_0^1
    \left[
      \int_{(0, 1)^d}
	(w_j - u_j)(w_k - u_k)
	\ddot{C}_{jk} (\vect u + st(\vect w - \vect u))
      \diff \mu_{n, \vect u}(\vect w)
    \right]
  \diff s.
\end{multline*}
It is sufficient to show that the absolute value of the integral in square brackets is $\Oh(\eps_n)$, uniformly in $j,k \in \{1, \ldots, d\}$ and $s, t \in (0, 1)$ and $\vect u \in [\eps_n, 1-\eps_n]^d$.

We apply the bound in \eqref{eq:second} to $\ddot{C}_{jk}( \vect u + st (\vect w - \vect u) )$. We have $\min(a^{-1}, b^{-1}) \le (ab)^{-1/2}$, and the latter is a convex function of $(a, b) \in (0, \infty)^2$. The point $\vect u + st (\vect w - \vect u)$ is located on the line segment connecting $\vect u$ and $\vect w$. Therefore,
\begin{align*}
  \bigl\lvert \ddot{C}_{jk}( \vect u + st (\vect w - \vect u) ) \bigr\rvert
  \le
  K \left[ \frac{1}{\{u_j(1-u_j)u_k(1-u_k)\}^{1/2}} + \frac{1}{\{w_j(1-w_j)w_k(1-w_k)\}^{1/2}} \right].
\end{align*}
We obtain
\begin{multline*}
  \left\lvert
    \int_{(0, 1)^d}
      (w_j - u_j)(w_k - u_k)
      \ddot{C}_{jk} (\vect u + st(\vect w - \vect u))
    \diff \mu_{n, \vect u}(\vect w)
  \right\rvert \\
  \le
  K 
  \int_{(0, 1)^d}
    \biggl[
    \frac{\lvert (w_j - u_j) (w_k - u_k) \rvert}{\{u_j(1-u_j)u_k(1-u_k)\}^{1/2}}
    +
    \frac{\lvert (w_j - u_j) (w_k - u_k) \rvert}{\{w_j(1-w_j)w_k(1-w_k)\}^{1/2}}
    \biggr]
  \diff \mu_{n, \vect u}(\vect w).
\end{multline*}
First, by the Cauchy--Schwarz inequality and the fact that $\Exp[ (S_i/n - u_i)^2 ] = n^{-1} u_i(1-u_i)$ for all $i \in \{1, \ldots, d\}$, we have
\[
  \int_{(0, 1)^d}
    \frac{\lvert (w_j - u_j) (w_k - u_k) \rvert}{\{u_j(1-u_j)u_k(1-u_k)\}^{1/2}}
  \diff \mu_{n, \vect u}(\vect w)
  \le
  \prod_{i \in \{j,k\}} \left\{ \int_{(0, 1)^d} \frac{(w_i - u_i)^2}{u_i(1-u_i)} \diff \mu_{n, \vect u}(\vect w) \right\}^{1/2}
  \le n^{-1} \le \eps_n.
\]
Second, again by Cauchy--Schwarz inequality,
\[
  \int_{(0, 1)^d}
    \frac{\lvert (w_j - u_j) (w_k - u_k) \rvert}{\{w_j(1-w_j)w_k(1-w_k)\}^{1/2}}
  \diff \mu_{n, \vect u}(\vect w)
  \le
  \prod_{i \in \{j,k\}} \left\{ \int_{(0, 1)^d} \frac{(w_i - u_i)^2}{w_i(1-w_i)} \diff \mu_{n, \vect u}(\vect w) \right\}^{1/2}.
\]
Each of the two integrals ($i = j$ and $i = k$), and therefore their geometric mean, will be bounded by the same quantity. Note that $\tfrac{1}{w_i(1-w_i)} = \tfrac{1}{w_i} + \tfrac{1}{1-w_i}$ and that the integral involving $\tfrac{1}{1-w_i}$ is equal to the one involving $\tfrac{1}{w_i}$ when $u_i$ is replaced by $1-u_i$, which we are allowed to do since $\vect u \in [\eps_n, 1-\eps_n]^d$ anyway. Therefore, we can replace $w_i(1-w_i)$ by $w_i$ in the denominator at the cost of a factor two. Further,
\begin{align*}
  \int_{(0, 1)^d} \frac{(w_i - u_i)^2}{w_i} \diff \mu_{n, \vect u}(\vect w)
  &\le
  \int_{[0, 1]^d} \ind_{(0, 1]}(w_i) \frac{(w_i - u_i)^2}{w_i} \diff \mu_{n, \vect u}(\vect w) \\
  &=
  \int_{[0, 1]^d} \ind_{(0, 1]}(w_i) \bigl( w_i - 2u_i + \tfrac{u_i^2}{w_i} \bigr) \diff \mu_{n, \vect u}(\vect w) \\
  &=
  u_i - 2u_i \Prob[S_i/n > 0] + u_i^2 \Exp[ \tfrac{1}{S_i/n} \ind_{\{ S_i/n > 0 \}} ] \\
  &\le
  -u_i + 2\Prob[S_i = 0] + n u_i^2 \Exp[ \tfrac{1}{S_i} \ind_{\{ S_i \ge 1 \}} ].
\end{align*}
Recall that $u_i \in [\eps_n, 1-\eps_n]$ and thus $\Prob[S_i = 0] \le (1-\eps_n)^n \le \exp(-n\eps_n) = n^{-1} = \oh(\eps_n)$. Further, the expectation of the reciprocal of a binomial random variable is treated in Lemma~\ref{lem:binom}. Note that $n \eps_n = \log n \to \infty$. We find
\[
  \sup_{\vect u \in [\eps_n, 1 - \eps_n]^d} \max_{i=1,\ldots,d} 
  \int_{(0, 1)^d} \frac{(w_i - u_i)^2}{w_i(1-w_i)} \diff \mu_{n, \vect u}(\vect w)
  = \Oh(n^{-1}) = \oh(\eps_n), \qquad n \to \infty.
\]
The proof is complete.
\end{proof}

\begin{lemma}
\label{lem:int}
For any $1/\{2(1-\omega)\} < \gamma < 1$, we have
\[
  \sup_{\vect u \in \{g \ge n^{-\gamma}\}} 
  \left\lvert 
    \int_{[0,1]^d} 
      \frac{g(\vect w)^\omega}{g(\vect u)^\omega} 
    \diff \mu_{n,\vect u} (\vect w) 
    - 1 
  \right\rvert 
  =
 \Oh \left\{ n^{-(1-\gamma)/2} \log(n) \right\},
 \qquad n \to \infty.
\]
\end{lemma}

\begin{proof}
Since $g(\frac{S_1}{n}, \dots, \frac{S_d}{n})$ is a random variable taking values in $[0, 1]$, we can write
\begin{align}
\nonumber
  \int_{[0,1]^d} 
      \frac{g(\vect w)^\omega}{g(\vect u)^\omega} 
  \diff \mu_{n,\vect u} (\vect w) 
  &=
  \frac{1}{g(\vect u)^\omega}
  \Exp \Bigl[ 
    g \bigl( \tfrac{S_1}{n}, \dots, \tfrac{S_d}{n} \bigr)^\omega
  \Bigr] \\
\label{eq:int:aux}
  &=
  \frac{1}{g(\vect u)^\omega}
  \int_0^1
    \Prob \Bigl\{
      g \bigl(\tfrac{S_1}{n}, \dots, \tfrac{S_d}{n} \bigr) > t^{1/\omega}
    \Bigr\}
  \diff t.
\end{align}
Split the integral into two pieces, $\int_0^{a_{n, \pm}} + \int_{a_{n, \pm}}^1$, where $a_{n, \pm} = a_{n, \pm}(\vect u) = g(\vect u)^\omega (1 \pm \eps_n)^\omega$. Write $\eps_n = n^{-(1-\gamma)/2} \log n$. Recall that $0 < \omega < 1/2$.

On the one hand, we find
\begin{align*}
  \int_{[0,1]^d} 
      \frac{g(\vect w)^\omega}{g(\vect u)^\omega} 
  \diff \mu_{n,\vect u} (\vect w) 
  &\le
  \frac{a_{n, +}}{g(\vect u)^\omega}
  +
  \frac{1 - a_{n,+}}{g(\vect u)^\omega} 
  \Prob \Bigl\{
    g \bigl(\tfrac{S_1}{n}, \dots, \tfrac{S_d}{n} \bigr) > a_{n,+}^{1/\omega}
  \Bigr\} \\
  &\le
  (1 + \eps_n)^\omega
  + 
  g( \vect u )^{-\omega}
  \Prob \Bigl\{
    g \bigl(\tfrac{S_1}{n}, \dots, \tfrac{S_d}{n} \bigr) 
    > g(\vect u) (1 + \eps_n)
  \Bigr\} \\
  &\le
  1 + \eps_n
  +
  g( \vect u )^{-\omega}
  2d \exp \{ - n g(\vect u) h(1+\eps_n) \}
\end{align*}
where we used \eqref{eq:upper} in the last step. Since $h(1+\eps_n) \ge \frac{1}{3} \eps_n^2$ for $0 \le \eps_n \le 1$ and since $g(\vect u) \ge n^{-\gamma}$, the upper bound is bounded by 
\[
  1 + \eps_n + 2d n^{\gamma \omega} \exp \{ - \tfrac{1}{3} (\log n)^2 \}
  =
  1 + \eps_n + \oh( \eps_n ),
  \qquad n \to \infty.
\]

On the other hand, restricting the integral in \eqref{eq:int:aux} to $[0, a_{n,-}]$, we have
\begin{align*}
  \int_{[0,1]^d} 
      \frac{g(\vect w)^\omega}{g(\vect u)^\omega} 
  \diff \mu_{n,\vect u} (\vect w) 
  &\ge
  \frac{a_{n,-}}{g(\vect u)^\omega}
  \Prob \Bigl\{
    g \bigl( \tfrac{S_1}{n}, \ldots, \tfrac{S_d}{n} \bigr)
    >
    a_{n,-}^{1/\omega}
  \Bigr\} \\
  &=
  (1 - \eps_n)^\omega
  \Prob \Bigl\{
    g \bigl( \tfrac{S_1}{n}, \ldots, \tfrac{S_d}{n} \bigr)
    >
    g(\vect u) (1 - \eps_n)
  \Bigr\} \\
  &\ge
  (1-\eps_n)^\omega [1 - 4d\exp \{ - n g(\vect u) h(1+\eps_n) \}],
\end{align*}
where we used \eqref{eq:lower} in the last step. Since $0 \le \eps_n \to 0$ and $g(\vect u) \ge n^{-\gamma}$, the lower bound is bounded from below by
\begin{align*}
  (1-\eps_n)^\omega [1 - 4d\exp \{ - n g(\vect u) h(1+\eps_n) \}]
  &\ge (1-\eps_n) [1 - 4d \exp \{ - \tfrac{1}{3} (\log n)^2 \} ] \\
  &\ge 1 - \eps_n - \oh( \eps_n ),
  \qquad n \to \infty.
\end{align*}
\end{proof}

\begin{lemma}
\label{lem:bias2}
As $n \to\infty$, we have, for any $\gamma \in (1/(2(1-\omega)),1)$
\[
  \sup_{\vect u \in \{g \ge n^{-\gamma}\}}  
    \left\lvert   
      \int_{[0,1]^d} 
	\biggl\{ 
	  \frac{\hat {\Cb}_n(\vect w)}{g(\vect w )^\omega} 
	  - 
	  \frac{\hat{\Cb}_n(\vect u)}{g(\vect u)} 
	\biggr\} 
	\frac{g(\vect w)^\omega}{g(\vect u)^\omega}
      \diff \mu_{n,\vect u} (\vect w)
    \right\rvert
  = \oh_\Prob(1).
\]
\end{lemma}

\begin{proof}
Let $\delta_n = 1 / \log(n)$. Write $\lVert \hat{\Cb}_n \rVert_\infty = \sup \{ \lvert \hat{\Cb}_n( \vect v) \rvert : \vect v \in [0, 1]^d \}$. We have $\lVert \hat{\Cb}_n \rVert_\infty = \Oh_p(1)$ as $n \to \infty$ by weak convergence of $\hat{\Cb}_n$ in $\ell^\infty([0, 1]^d)$.

We split the integral over $\vect w \in [0, 1]^d$ into two pieces: the integral over the domain
\[ 
  A_{n, \vect u} = 
  \{ \vect w \in [0, 1]^d : \lvert \vect w - \vect u \rvert_\infty > \delta_n \} \cup \{ g < n^{-\gamma}(1-\delta_n) \}
\]
and the integral over its complement; here $\lvert \vect x \rvert_\infty = \max\{ \lvert x_j \rvert : j = 1, \ldots, d \}$.

For all $\vect w \in [0, 1]^d$ and all $\vect u \in \{g \ge n^{-\gamma} \}$, we have
\begin{align*}
  R_n(\vect u, \vect w)
  := 
  \biggl\lvert
    \frac{\hat{\Cb}_n(\vect w)}{g(\vect w)^\omega} 
    - 
    \frac{\hat{\Cb}_n(\vect u)}{g(\vect u)^\omega} 
  \biggr\rvert
  \frac{g(\vect w)^\omega}{g(\vect u)^\omega} 
  \le
  \frac{\lvert \hat{\Cb}_n(\vect w) \rvert}{g(\vect u)^\omega}
  +
  \frac{ \lvert \hat{\Cb}_n(\vect u) \rvert}{g(\vect u)^{2\omega}} 
  \le
  2 \lVert \hat{\Cb}_n \rVert_\infty n^{2\gamma\omega}.
\end{align*}
Moreover, for all $\vect u \in \{g \ge n^{-\gamma}\}$, using Chebyshev's inequality and the concentration inequality \eqref{eq:lower}, we have
\begin{align*}
  \mu_{n,\vect u} \bigl( A_{n, \vect u} \bigr)
  &\le
  \sum_{j=1}^d \Prob \Bigl\{ \bigl\lvert \tfrac{S_j}{n} - u_j \bigr\rvert > \delta_n \Bigr\}
  +
  \Prob \Bigl\{
    g \bigl( \tfrac{S_1}{n}, \ldots, \tfrac{S_d}{n} \bigr)
    < g(\vect u) (1 - \delta_n)
  \Bigr\} \\
  &\le
  d n^{-1} \delta_n^{-2} + 4d \exp \{ - n^{1-\gamma} h(1+\delta_n) \}.
\end{align*}
Since $0 < \omega < 1/2$, $0 < \gamma < 1$, $\delta_n = 1/\log(n)$ and $h(1+\delta_n) \ge \tfrac{1}{3} \delta_n^2$, it follows that
\[
  \sup_{\vect u \in \{g \ge n^{-\gamma}\}}
  \int_{A_{n, \vect u}}
    R_n( \vect u, \vect w )
  \diff \mu_{n, \vect u}( \vect w )
  \le
  n^{2\gamma\omega} [n^{-1} \delta_n^{-2} + \exp \{ - n^{1-\gamma} h(1+\delta_n) \}] \, \Oh_p(1) 
  =
  \oh_p(1), \qquad n \to \infty.
\]
It remains to consider the integral over $\vect w \in [0, 1]^d \setminus A_{n, \vect u}$, i.e., $\lvert \vect w - \vect u \rvert_\infty \le \delta_n$ and $g( \vect w ) \ge n^{-\gamma}(1 - \delta_n) > n^{-1}$, at least for sufficiently large $n$. By Lemma~4.1 in \cite{BerBucVol17}, we have
\begin{equation}
\label{eq:modcont}
  \sup_{\substack{ \vect u, \vect w \in \{ g \ge n^{-1} \} \\ \lvert \vect u - \vect w \rvert_\infty \le \delta_n }}
  \left|
    \frac{\hat{\Cb}_n(\vect w)}{g(\vect w )^\omega} 
    - 
    \frac{\hat{\Cb}_n(\vect u)}{g(\vect u)} 
  \right|
  = \oh_{\Prob}(1), \qquad n \to \infty.
\end{equation}
In view of Lemma~\ref{lem:int}, we obtain that
\[
  \sup_{\vect u \in \{g \ge n^{-\gamma}\}}
  \int_{[0, 1]^d \setminus A_{n, \vect u}}
    R_n( \vect u, \vect w )
  \diff \mu_{n, \vect u}( \vect w )
  \le
  \oh_{\Prob}(1) \int_{[0, 1]^d} \frac{g(\vect w)^\omega}{g(\vect u)^\omega} \diff \mu_{n, \vect u}(\vect w)
  = \oh_{\Prob}(1),
\]
as $n \to \infty$. The stated limit relation follows by combining the assertions on the integral over $A_{n, \vect u}$ and the one over its complement.

Note that in Lemma~4.1 in \cite{BerBucVol17}, the supremum in \eqref{eq:modcont} is taken over $[1/n,1-1/n]^d$ instead of over $\{g > n^{-1}\}$. But it can be seen in the proof of that statement that the result can be extended to the set $\{ g \geq n^{-1} \}$. Furthermore, in the latter reference, the pseudo-observations are defined as $\hat U_{i,j}=\frac{1}{n+1} R_{i,j}$. However, this does not affect the above proof, since the difference of the two empirical copulas is at most $d/n$, almost surely. This gives an additional error term on the event $\{ g \geq n^{-1} \}$ which is of the order $\Oh_\Prob(n^{\omega+1/2-1}) = \oh_\Prob(1)$, as $n \to \infty$.
\end{proof}

\subsection{On the expectation of the reciprocal of a binomial random variable}

\begin{lemma}
Let $0 < u \le 1$ and let $n \ge 2$ be integer. If $S \sim \Bin(n, u)$ and $T \sim \Bin(n-1, u)$, then
\begin{equation}
\label{eq:Bin:invMoment}
  \Exp \left[ \frac{1}{S} \ind_{ \{ S \ge 1 \} } \right]
  =
  nu \Exp \left[ \frac{1}{(1+T)^2} \right]
  =
  nu \int_0^1 (1-u+us)^{n-1} (-\log s) \, \diff s.
\end{equation}
\end{lemma}

\begin{proof}
For $k\in\{1,\dots,n\}$, we have
\[
  \frac{\Prob(S=k)}{\Prob(T+1=k)}
  =
  \frac{\binom{n}{k}u^k(1-u)^{n-k}}{\binom{n-1}{k-1}u^{k-1}(1-u)^{(n-1)-(k-1)}}
  =
  \frac{nu}{k}.
\]
We obtain that
\begin{align*}
  \Exp \left[ \frac{1}{S} \ind_{\{S\geq 1\}} \right] 
  =
  \sum_{k=1}^n \frac{1}{k} \Prob(S = k)
  =
  \sum_{k=1}^n \frac{1}{k} \frac{nu}{k} \Prob(T+1=k)
  =
  (nu) \Exp \left[ \frac{1}{(1+T)^2} \right].
\end{align*}
Now we apply a trick due to \cite{ChaStr72}: we have
\[
  \frac{1}{(1+T)^2}
  =
  \int_{t=0}^1 \frac{1}{t} \int_{s=0}^t s^T \, \diff s \, \diff t
  =
  \int_{s=0}^1 s^T \int_{t=s}^1 \frac{\diff t}{t} \, \diff s
  =
  \int_0^1 s^T (- \log s) \, \diff s.
\]
Taking expectations and using Fubini's theorem, we obtain
\[
  \Exp \left[ \frac{1}{(1+T)^2} \right]
  =
  \int_0^1 \Exp(s^T) \, (- \log s) \, \diff s
  =
  \int_0^1 (1-u+us)^{n-1} (-\log s) \, \diff s,
\]
as required.
\end{proof}

\begin{lemma}
\label{lem:binom:aux}
Let $0 < u_n \le 1$ and let $S_n \sim \Bin(n, u_n)$. If $nu_n \to \infty$, then
\[
  (nu_n^2) \Exp \left[ \frac{1}{S_n} \ind_{ \{ S_n \ge 1 \} } \right]
  =
u_n+ \Oh \bigl( n^{-1} \bigr),
  \qquad n \to \infty.
\]
\end{lemma}

\begin{proof}
We start from \eqref{eq:Bin:invMoment}:
\[
  (nu_n^2) \Exp \left[ \frac{1}{S_n} \ind_{ \{ S_n \ge 1 \} } \right]
  =
  n^2 u_n^3 \int_0^1 (1-u_n+u_ns)^{n-1} (-\log s) \, \diff s.
\]
We split the integral in two parts, cutting at $s = 1/2$.

First we consider the case $s \le 1/2$. For some positive constant $K$, we have
\begin{align*}
  n^2u_n^3 \int_0^{1/2} (1-u_n+u_ns)^{n-1} (-\log s) \, \diff s
  &\le
  n^2u_n^3 \, (1-u_n/2)^{n-1} \int_0^{1/2} (-\log s) \, \diff s \\
  &\le
  K \,   n^2u_n^3  \, (1-u_n/2)^n \\
  &\le
  K \,   n^2u_n^3  \exp(-nu_n/2).
\end{align*}
For any $m > 0$, this expression is $\Oh(u_n (nu_n)^{-m}) = \oh(n^{-1})$ as $n \to \infty$, hence by choosing $m=1$ it is $\Oh(n^{-1})$ as $n \to \infty$.

Second we consider the case $s \ge 1/2$. The substitution $s = 1 - v / (nu_n)$ yields
\begin{equation}
\label{eq:Bin:invMoment:aux}
  n^2 u_n^3 \int_{1/2}^1 (1-u_n+u_ns)^{n-1} (-\log s) \, \diff s
  =
  u_n  \int_0^{(nu_n/2)} (1-v/n)^{n-1} [-(nu_n) \log \{ 1 - v/(nu_n) \}] \, \diff v.
\end{equation}
We need to show that this integral is $u_n + \Oh(n^{-1})$ as $n \to \infty$.

For facility of writing, put $k_n = nu_n$. Recall that $k_n \to \infty$ as $n \to \infty$ by assumption. The inequalities $x \le - \log(1-x) \le x/(1-x)$ for $0 \le x < 1$ imply that
\[
  0 \le -k_n \log(1 - v/k_n) - v \le \frac{v^2}{k_n-v} \le \frac{2 v^2}{k_n}, \qquad v \in [0, k_n/2].
\]
As $(1-v/n)^{n-1} \le (1-k_n/(2n))^{-1} (1-v/n)^{-n} \le 2 \exp(-v)$ for $v \in [0, k_n/2]$, we find
\begin{align*}
 u_n \int_0^{k_n/2} (1-v/n)^{n-1} \, \left\lvert -k_n \log ( 1 - v/k_n ) - v \right\rvert \, \diff v
  &\le
  \frac{4 u_n}{k_n} \int_0^{k_n/2} \exp(-v) \, v^2 \, \diff v \\
  &=
  \Oh( n^{-1} ),
  \qquad n \to \infty.
\end{align*}
As a consequence, replacing $-k_n \log(1-v/k_n)$ by $v$ in \eqref{eq:Bin:invMoment:aux} produces an error of the required order $\Oh(n^{-1})$.

It remains to consider the integral
\[
  u_n\int_0^{k_n/2} (1-v/n)^{n-1} \, v \, \diff v.
\]
Via the substitution $x = 1-v/n$, this integral can be computed explicitly. After some routine calculations, we find it is equal to
\[
 u_n  \frac{n}{n+1} [ 1 - \{1 - k_n/(2n)\}^{n} (1 + k_n/2) ].
\]
Since $\{1 - k_n/(2n)\}^{n} \le \exp(-k_n/2)$, the previous expression is
\[
u_n + \Oh(u_n n^{-1}) + \Oh(u_n \exp(-k_n/2) k_n ), \qquad n \to \infty,
\]
The error term is $\Oh(n^{-1})$, as required.
\end{proof}

%

\begin{lemma}
\label{lem:binom}
If $0 < u_n \le 1$ is such that $nu_n \to \infty$ as $n \to \infty$, then
\[
  \sup_{u_n \le u \le 1}  \left\lvert nu^2 \Exp \left[ \frac{1}{S} \ind_{ \{ S \ge 1 \} } \right] - u \right\rvert
  =
  \Oh( n^{-1} ),
  \qquad n \to \infty,
\]
where the expectation is taken for $S \sim \Bin(n,u)$.
\end{lemma}

\begin{proof}
The function sending $u \in [u_n, 1]$ to $\lvert nu^2 \Exp [ S^{-1} \ind_{ \{ S \ge 1 \} } ] - 1 \rvert$, with $S \sim \Bin(n, u)$, is continuous and therefore attains its supremum at some $v_n \in [u_n, 1]$. Since $nv_n \ge nu_n \to \infty$ as $n \to \infty$, we can apply Lemma~\ref{lem:binom:aux} to find that the supremum is $\Oh(n^{-1})$ as $n \to \infty$.
\end{proof}

\subsection{Inequalities for binomial random variables}

If $S \sim \Bin(n, u)$ is a binomial random variable with succes probability $0 < u < 1$, then Bennett's inequality states that
\[
  \Pr \Bigl(
    \sqrt{n} \bigl\lvert \tfrac{S}{n} - u \bigr\rvert \ge \lambda
  \Bigr)
  \le
  2 \exp \Bigl\{
    - \frac{\lambda^2}{2u} 
    \psi \Bigl( \frac{\lambda}{\sqrt{n} u} \Bigr)
    \Bigr\}
  =
  2 \exp \Bigl\{
    - nu \, h \Bigl( 1 + \frac{\lambda}{\sqrt{n} u} \Bigr)
    \Bigr\}
\]
for $\lambda > 0$, where $\psi(x) = 2 \, h(1+x)/x^2$ and $h(x) = x(\log x - 1) + 1$; see for instance \citet[Proposition~A.6.2]{VanWel96}. Setting $\lambda = \sqrt{n} u \delta$, we find
\begin{equation}
\label{eq:bennett}
  \Pr \Bigl(
    \big\lvert \tfrac{S}{n} - u \big\rvert \ge u \delta
  \Bigr)
  \le
  2 \exp \{ - nu \, h(1+\delta) \},
  \qquad \delta > 0.
\end{equation}
Note that $h(1+\delta) = \int_0^\delta \log(1+t) \diff t \ge \int_0^\delta (t - \tfrac{1}{2} t^2) \diff t = \tfrac{1}{2} \delta^2 (1 - \tfrac{1}{3} \delta)$ for $\delta \ge 0$ and thus $h(1+\delta) \ge \tfrac{1}{3} \delta^2$ for $0 \le \delta \le 1$.
We extend \eqref{eq:bennett} to a vector of independent binomial random variables and in terms of the weight function $g$ in \eqref{eq:g}.

\begin{lemma}
\label{lem:bounds}
If $S_1, \dots, S_d$ are independent random variables with $S_j \sim \Bin(n,u_j)$ and $0 < u_j < 1$ for all $j \in \{1, \ldots, d\}$, then, for $\delta > 0$,
\begin{align}
\label{eq:upper}
  \Prob \Bigl\{ 
    g \bigl( \tfrac{S_1}{n}, \ldots, \tfrac{S_d}{n} \bigr) 
    \geq g(\vect u ) ( 1 + \delta ) 
  \Bigr\}
  &\leq 
  2d \exp\bigl\{ - n g(\vect u) h(1+\delta) \bigr\}, \\
\label{eq:lower}
  \Prob \Bigl\{ 
    g \bigl( \tfrac{S_1}{n}, \ldots, \tfrac{S_d}{n} \bigr) 
    \leq g(\vect u ) ( 1 - \delta )
  \Bigr\}
  &\leq
  4d \exp \bigl\{ - n g(\vect u) h(1+\delta) \bigr\},
\end{align}
with $h$ as above; in particular, $h(1+\delta) \ge \tfrac{1}{3} \delta^2$ for $0 < \delta \le 1$.
\end{lemma}


\begin{proof}
Let us start with \eqref{eq:lower}. The definition of the weight function $g$ in \eqref{eq:g} yields
\[
  \Prob 
  \Bigl\{ 
    g \bigl( \tfrac{S_1}{n}, \dots , \tfrac{S_d}{n} \bigr) 
    \leq g(\vect u ) ( 1 - \delta )  
  \Bigr\}
  \leq  
  \sum_{j=1}^d 
  \Bigl[ 
    \Prob 
    \Bigl\{ 
      \tfrac{S_j}{n} \leq g(\vect u ) (1 - \delta) 
    \Bigr\} 
    +
    \Prob 
    \Bigl\{ 
      \max_{k \ne j} \bigl( 1 - \tfrac{S_{k}}{n} \bigr)
      \leq g(\vect u ) ( 1 - \delta) 
    \Bigr\} 
  \Bigr].
\]
Let us first consider the first term on the right-hand side, i.e., $\Prob \{ \tfrac{S_j}{n} \leq g(\vect u ) ( 1 - \delta ) \}$. By definition of the weight function we have $g (\vect u) \leq u_j$. By Bennett's inequality~\eqref{eq:bennett},
\begin{align*}
  \Prob \Bigl\{ \tfrac{S_j}{n} \leq g(\vect u) ( 1 - \delta ) \Bigr\}  
  &\leq 
  \Prob \Bigl\{ \tfrac{S_j}{n} \leq u_j ( 1 - \delta ) \Bigr\} \\
  &\le
  \Prob \Bigl\{ \big\lvert \tfrac{S_j}{n} - u_j \big\rvert 
    \geq u_j \delta \Bigr\} \\ 
  &\leq
  2 \exp \{ - n u_j h(1 + \delta) \} \\ 
  &\leq
  2 \exp \{ - n g(\vect u) h(1 + \delta) \} .
\end{align*}
Second, consider the term $\Prob \{ \max_{k \ne j} ( 1 - \tfrac{S_{k}}{n} )\leq g(\vect u ) ( 1 - \delta) \} $. Suppose $j=1$; the other cases can be treated exactly along the same lines. We have $g(\vect u) \leq \max_{k \neq 1}(1 - u_k)$. Assume without loss of generality that $\max_{k \neq 1} (1 - u_k) = 1 - u_2$. Then we obtain $g(\vect u) \le 1 - u_2$ and, by Bennett's inequality \eqref{eq:bennett} applied to $n - S_2 \sim \Bin(n, 1-u_2)$,
\begin{align*}
  \Prob
  \Bigl\{ 
    \max_{k \ne 1} \bigl(1-\tfrac{S_{k}}{n}\bigr)
    \leq g(\vect u ) (1 - \delta) 
  \Bigr\} 
  &\leq 
  \Prob
  \Bigl\{ 
    \max_{k \ne 1} \bigl(1-\tfrac{S_{k}}{n}\bigr)
    \leq (1 - u_2) (1 - \delta) 
  \Bigr\} \\
  &\leq
  \Prob 
  \Bigl\{ 
    1 - \tfrac{S_{2}}{n}
    \leq (1- u_2) (1 - \delta)
  \Bigr\} \\
  &\leq
  \Prob 
  \Bigl\{
    \big\lvert 1 - \tfrac{S_2}{n} - (1- u_2) \big\rvert 
    \geq (1- u_2 ) \delta 
  \Bigr\} \\
  &\leq
  2 \exp \{ -n (1-u_2) h(1 + \delta) \} \\
  &\leq
  2 \exp \{ -n g(\vect u) h(1 + \delta) \}.
\end{align*}

Let us now show \eqref{eq:upper}. First suppose $g(\vect u) = u_1$. Since $g(\frac{S_1}{n} , \dots \frac{S_d}{n }) \leq \frac{S_1}{n} $ we have, by Bennett's inequality~\eqref{eq:bennett}, 
\begin{align*}
  \Prob 
  \Bigl\{ 
    g \bigl( \tfrac{S_1}{n}, \dots, \tfrac{S_d}{n} \bigr) 
    \geq g(\vect u) (1 + \delta) 
  \Bigr\} 
  &\leq 
  \Prob 
  \Bigl\{ 
    \tfrac{S_1}{n} \geq u_1 (1 + \delta) 
  \Bigr\} \\
  &\le
  \Prob
  \Bigl\{
    \bigl\lvert \tfrac{S_1}{n} - u_1 \bigr\rvert \ge u_1 \delta
  \Bigr\} \\
  &\le
  2 \exp \bigl\{ - nu_1 h(1+\delta) \bigr\}
  =
  2 \exp \bigl\{ - n g(\vect u) h(1+\delta) \bigr\}.
\end{align*}
Finally, suppose that $g(\vect u) = 1 - u_1 \geq 1 - u_k$, for $k = 3, \ldots, d$. Note that $g (\tfrac{S_1}{n}, \dots, \tfrac{S_d}{n }) \leq \max_{k \neq 2} (1 - \tfrac{S_j}{n}) $, which yields 
\begin{align*}
  \Prob \Bigl\{ 
    g \bigl( \tfrac{S_1}{n}, \ldots, \tfrac{S_d}{n} \bigr) 
    \geq g(\vect u ) (1 + \delta)
  \Bigr\}
  &\leq 
  \Prob \Bigl\{ 
    \max_{k \neq 2} (1- \tfrac{S_k}{n}) \geq (1- u_1) (1 + \delta)
  \Bigr\} \\
  &\leq 
  \sum_{k \neq 2} \Prob \Bigl\{ 
    1 - \tfrac{S_k}{n} \geq (1- u_1)(1 + \delta) 
  \Bigr\}.
\end{align*}
By Bennett's inequality \eqref{eq:bennett} applied to $n - S_k \sim \Bin(n, 1-u_k)$ for every $k \ne 2$, we have, since $(1-u_1)/(1-u_k) \ge 1$,
\begin{align*}
  \Prob \Bigl\{ 
    1 - \tfrac{S_k}{n} \geq (1- u_1)(1 + \delta) 
  \Bigr\}
  &\le
  \Prob \Bigl\{
    \bigl\lvert 1 - \tfrac{S_k}{n} - (1-u_k) \bigr\rvert
    \geq
    (1 - u_1)(1+\delta) - (1 - u_k)
  \Bigr\} \\
  &\le
  2 \exp \Bigl\{
    - n (1 - u_k) h \Bigl( \tfrac{1 - u_1}{1 - u_k} (1 + \delta) \Bigr)
  \Bigr\}.
\end{align*}
For $a \ge 1$ and $\delta \ge 0$, a direct calculation\footnote{Or, since $h(x) = \int_1^x \log(t) \, \diff t$, we have $h(a(1+\delta)) = \int_1^{a(1+\delta)} \log(t) \, \diff t = a \int_{1/a}^{1+\delta} \log(as) \, \diff s \ge a \int_1^{1+\delta} \log(s) \, \diff s = a \, h(1+\delta)$ for $a \ge 1$ and $\delta \ge 0$.} shows that $h(a(1+\delta)) - a \, h(1+\delta) \ge h(a) \ge 0$ and thus $h(a(1+\delta)) \ge a \, h(1+\delta)$. Apply this inequality to $a = (1-u_1) / (1-u_k)$ to find
\begin{align*}
  \Prob \Bigl\{ 
    1 - \tfrac{S_k}{n} \geq (1- u_1)(1 + \delta) 
  \Bigr\}
  &\le
  2 \exp \Bigl\{ - n (1 - u_k) \tfrac{1 - u_1}{1 - u_k} h(1+\delta) \Bigr\} \\
  &=
  2 \exp \{ - n (1 - u_1) h(1+\delta) \}
  =
  2 \exp \{ - n g(\vect u) h(1+\delta) \}. \qedhere
\end{align*}
\end{proof}

\subsection{Extensions of results in \cite{BerBucVol17}}
\label{sec:BerBucVol17}

For any sequence $\delta_n>0 $ that converges to zero as $n \to \infty$, Lemma~4.10 in \cite{BerBucVol17} can be extended to 
\begin{equation}
\label{eq:extend1}
  \sup \left\{ 
    \left|
      \frac{\Cb_n(\vect u )}{g(\vect u)^\omega} 
      - 
      \frac{\Cb_n(\vect u')}{g(\vect u')^\omega} 
    \right|
    \; : \;
    g(\vect u) \ge c/n, \, 
    g(\vect u') \ge c/n, \,
    \lvert  \vect  u - \vect u' \rvert \leq \delta_n
  \right\}
  = \oh_\Prob(1), \qquad n \to \infty.
\end{equation}
Here, $\Cb_n = \sqrt{n} ( \tilde{C}_n - C )$ and $\tilde{C}_n$ is the empirical copula based on the generalized inverse function of the marginal empirical distribution functions \citep[beginning of Section~4.2]{BerBucVol17}.
Furthermore, Theorem~4.5 in the same reference can be extended to
\begin{equation}
\label{eq:extend2}
  \sup \left\{ 
    \left| 
      \frac{\hat \Cb_n(\vect u)}{g(\vect u)^\omega} - \frac{\bar \Cb_n(\vect u)}{g(\vect u )^\omega} 
    \right|
    \; : \;
    g(\vect u) \geq c/n
  \right\}
  = \oh_\Prob(1), \qquad n \to \infty.
\end{equation}

\begin{proof}
Let us start with \eqref{eq:extend1}. The result is similar to the result in Lemma 4.10, in particular Equation (4.1), in \cite{BerBucVol17}. A look at the proof of the result shows that the restriction $\vect u , \vect u' \in [c/n, 1-c/n]^d$ instead of $ \vect u , \vect u' \in \{ g \geq c/n \}$ is not needed. The proof of Equation (4.1) in Lemma 4.10 in \cite{BerBucVol17} is based on Lemma 4.7, 4.8 and Equations (4.8) and (4.8) which are all valid on sets of the form $N(c_{n1},c_{n2})= \{ g \in (c_{n1}, c_{n2}] \}$. Hence, in the proof, all suprema can be taken over $ \vect u , \vect u' \in \{ g \geq c/n \}$ instead of $\vect u , \vect u' \in [c/n, 1-c/n]^d$, which gives us exactly \eqref{eq:extend1}.

For the proof of \eqref{eq:extend2} note that for any $\vect u \in \{ g \geq c/n \}$ we can find $\vect u' \in \{ g \geq n^{-1/2} \}$ such that $\lvert \vect u - \vect u' \rvert \leq d n^{-1/2}$. To find such a $\vect u'$ is all that it is needed to extend the proof of Theorem 4.5 in \cite{BerBucVol17} to obtain \eqref{eq:extend2}.
\end{proof}

\section*{Acknowledgments}

The authors gratefully acknowledge the editor-in-chief, the associate editor, and the referees for additional references, for suggesting the idea of a weighted test of independence, and for various suggestions concerning the numerical experiments on the estimation of the Pickands dependence function.

Betina Berghaus gratefully acknowledges support by the Collaborative Research Center ``Statistical modeling of nonlinear dynamic processes'' (SFB 823) of the German Research Foundation (DFG). 

Johan Segers gratefully acknowledges funding by contract ``Projet d'Act\-ions de Re\-cher\-che Concert\'ees'' No.\ 12/17-045 of the ``Communaut\'e fran\c{c}aise de Belgique'' and by IAP research network Grant P7/06 of the Belgian government.

\bibliography{biblioN}
\end{document}